\documentclass[12pt]{article}

\usepackage{color}
\usepackage{geometry}                
\usepackage{graphicx}
\usepackage{amssymb}
\usepackage{epstopdf}
\usepackage{graphicx}
\usepackage{amssymb}
\usepackage{epstopdf}
\usepackage{amsmath}
\usepackage{amsthm}
\usepackage{amsfonts}
\usepackage{epstopdf}

\newtheorem{theorem}{Theorem}
\newtheorem{lemma}[theorem]{Lemma}

\newtheorem{corollary}[theorem]{Corollary}
\newtheorem{claim}[theorem]{Claim}

\newtheorem{proposition}[theorem]{Proposition}
\theoremstyle{definition}

\theoremstyle{remark}
\newtheorem{remark}[theorem]{Remark}
\numberwithin{equation}{section}

\newcommand\RR{{\mathbb R}}

\newcommand\EE{{\mathbb E}}
\newcommand\PP{{\mathbb P}}

\newcommand\ev{{\mathrm E}}
\newcommand\pr{{\mathrm P}}

\def\M{\mathcal{ M}}

\def\A{\mathcal{ A}}
\def\S{\mathcal{S}}
\def\W2{W^{1,2}({\cal O}(M))}

\def\E{\mathcal{ E}}

\def\G{\mathfrak{G}}

\newcommand\one{{\mathbf 1}}

\newcommand{\MM}{\mathcal M}
\newcommand{\stb}{\mbox{\small $\frac{2}{\sqrt{\beta}}$}}

\newcommand{\tb}{\mbox{\small $\frac{2}{\beta}$}}
\newcommand{\mm}{\mathfrak m}
\def\ss{\mathfrak s}
\newcommand{\tr}{\mathrm{tr}}

\title{Spiking the random matrix hard edge} 

\author{{Jos\'e A. Ram\'irez}\footnote{Department of Mathematics,  Universidad de Costa Rica. e-mail:
{\tt{alexander.ramirez$\_$g@ucr.ac.cr.}}},  { Brian Rider}\footnote{Department of Mathematics, Temple University. e-mail: {\tt{brian.rider@temple.edu.}}}}

\date{} 

\begin{document}
\maketitle

\begin{abstract}
We characterize the limiting smallest eigenvalue distributions (or hard edge laws) for sample covariance type matrices drawn from a spiked population. In the case of a single spike, the results are valid in the context of the general $\beta$ ensembles. For multiple spikes, the necessary construction restricts matters to real, complex or quaternion ($\beta=1,2,$ or $4$) ensembles.  The limit laws are described in terms of a random  integral operators, and partial differential equations satisfied by the corresponding distribution functions are derived as corollaries. We also show that, under a natural limit, all spiked hard edge laws derived here degenerate to the critically spiked soft edge laws (or deformed Tracy-Widom laws). The latter were first described at $\beta=2$ by Baik, Ben Arous, and Pech\'e \cite{BBP}, and from a unified $\beta$ random operator point of view by Bloemendal and Vir\'ag \cite{BV1, BV2}.
\end{abstract}


\section{Introduction}

A basic problem in mathematical statistics is to describe the law of the largest eigenvalue of sample covariance (or Wishart) matrices of the form $X \Sigma X^{\dagger}$ in which   $X$ is $n \times m$ and comprised  independent unit Gaussians in $\mathbb{F} = \RR, \mathbb{C}, $ or $\mathbb{H}$ and $\dagger$ is the associated conjugate transpose. One is typically interested in the limit as $n$ and $m$ tend to $\infty$ with $\Sigma$ some deterministic sequence of symmetric population matrices. When $\Sigma$ is the identity, this so-called soft-edge limit is well known to be given by the $\beta=1,2$ or $4$ Tracy-Widom laws  \cite{TW1, TW2} (for the case of   $\RR, \mathbb{C}, $ or $\mathbb{H}$ entries respectively). Moving toward the more general problem, the spiked ensembles introduced by Johnstone in which $\Sigma = \Sigma_r \oplus I_{m-r}$ and $r$ remains fixed as $n$ and $m$ grow have generated considerable interest.  Here the subscripts specify the dimensions of the diagonal matrix $\Sigma_r$ and identity matrix $I_{m-r}$.

Using the determinantal framework at $\beta=2$, \cite{BBP} proved there exists a phase transition (with respect to the entries of $\Sigma_r$).  Below criticality one sees Tracy-Widom in the limit, above criticality there  are Gaussian effects (the limit given by the law of the largest eigenvalue of an $r \times r$ GUE), with a new one parameter family of spiked soft-edge laws in the crossover regime.    Subsequent analytic work was carried at $\beta = 1$ and $\beta = 4$ in \cite{Mo} and \cite{Wang}.

 In another direction, Bloemendal-Vir\'ag \cite{BV1, BV2} proved that the $\beta=1,2,$ or $4$ soft-edge (critically) spiked laws can  be characterized in a unified way
through the eigenvalue problem for the random operator $\mathcal{H}$ acting on functions $f \in L^2[ [0,\infty), \mathbb{F}^r]$ defined by
\begin{equation}
\label{SAOr}
  \mathcal{H} = - \frac{d^2}{dx^2} + r x + \sqrt{2} \mathcal{B}_x^{\prime}, \qquad f^{\prime}(0) = W f(0).
\end{equation}
Here $\mathcal{B}_t$ is the standard $\mathbb{F}$-invariant Brownian motion, that is, for $U \in U_r(\mathbb{F}) =
\{ V \in \mathbb{F}^{r \times r} : V V^\dagger = I \}$ it holds that $U \mathcal{B}_t U^{\dagger} \sim  \mathcal{B}_t$ (equality in law),  and $W$ is the appropriate scaling limit of the matrix $\Sigma_r$. At $r=1$ the result holds for all $\beta > 0$. In that case the noise term reduces to $\frac{2}{\sqrt{\beta}} b_x^{\prime} $ for $x \mapsto b_x$ a standard Brownian motion, and $\mathcal{H}$ is recognized as the Stochastic Airy Operator from \cite{RRV}, but with a simple Robin $f'(0) = w f(0)$ rather than Dirichlet boundary condition at the origin. Hence, for $r=1$ the single  parameter $c$ interpolates between the general beta Tracy-Widom  law at $w =+\infty$ (which forces $f(0)=0$) and the Gaussian law in the supercritical regime, $w =-\infty$ (which requires a little perturbation argument to make precise, see \cite{BV1}).

The results of \cite{BV1, BV2} follow upon those of \cite{RRV} in that they hinge on tridiagonal representations for the spiked Wishart ensembles. For $r=1$ these are the same random tridiagonal matrices introduced by Dumitriu-Edelman \cite{DE}, which subsequently formed the basis for the conjectured Stochastic Airy Operator formulation of the Tracy-Widom($\beta$) laws in \cite{ES,Sutton} (and again proved in \cite{RRV}). In particular consider the $X \Sigma^{1/2}$ for $X$ as above (comprised of iid real/complex/quaternion Gaussians) but specify  
$\Sigma = \sigma \oplus I_{m-1}$ for a scalar $\sigma$. Then the classical Householder procedure can be carried out, taking $X \Sigma^{1/2}$ into 
\begin{equation}
\label{B} 
   B_{\sigma} = \frac{1}{\sqrt{\beta}} \left[ \begin{array}{ccccc}
\sqrt{\sigma} \chi_{m \beta} & \chi_{(n-1)\beta} & & & \\
& \chi_{(m-1)\beta} & \chi_{(n-2)\beta} & & \\
& & \ddots & \ddots & \\
& & & \chi_{(m-n+2)\beta} & \chi_{\beta} \\
& & & &\chi_{(m -n +1)\beta}  \\
\end{array}
\right],
\end{equation}
up to a unitary conjugation. That is, $B_{\sigma}^{} B_{\sigma}^T$ has the same eigenvalues as $X \Sigma X^{\dagger}$ in the real/complex/quaternion  ($\beta = 1/2/4$) cases, and afterwards this ensemble may be declared the general one-spiked model. For $r >1$, one of the key contributions of \cite{BV2} is a block tridiagonalization procedure (which we revisit in Section 3) that commutes with $r$-fold spiking.  Unfortunately, this model does not immediately generalize to all $\beta > 0$.

Our goal here is to establish an analogous set of results for the critically spiked ``hard edge". The hard edge refers to the limit laws for the minimal eigenvalues of Wishart type ensembles in the regime $m = n+a$ (which forces a non-trivial interaction of the random spectrum with the origin). Classically - that is, for $\beta=1,2,4$ -  these laws were also identified by Tracy and Widom in terms of Painlev\'e functions \cite{TW2}. They come indexed by the parameter $a$ and degenerate to the more well known Tracy-Widom($\beta = 1,2,4$) laws of \cite{TW1,TW3} in the limit $a \rightarrow \infty$. This hard-to-soft transition in the classical setting appears to have first 
been worked out in \cite{BF}. For general $\beta >0$ we introduced a random operator for the hard edge in \cite{RR} and demonstrated that, as expected, its ordered eigenvalues went over into those for the  Stochastic Airy Operator for all $\beta$.
 
To critically spike the hard edge the entries of $\Sigma_r$ must be tuned to zero as $n \rightarrow \infty$. We show that if $n \Sigma_r$ has an appropriate limit then there is a new family of random operator limits indexed by these spiking parameters. These spiked hard edge operators produce diffusion/PDE descriptions for the new limit laws (in the manner familiar from \cite{BV1, BV2, RR, RRV}), and, again as expected, recapture all known spiked Tracy-Widom laws in the $a \rightarrow \infty$ limit.

We mention that for $\beta =2 $ and at the level of correlation functions the hard edge spiking was worked out in \cite{DF}.

The next few articles of the introduction detail: (i) the limiting spiked hard edge operators, (ii) diffusion and partial differential equation descriptions of the corresponding spiked hard edge laws,  (iii) the transition between the spiked hard edge laws and the spiked/deformed Tracy-Widom laws (for all $\beta$ at $r=1$, and for $\beta=1, 2, 4$ otherwise), and last (iv) a description of the supercritical regime for the hard edge via a recent matrix generalization of the Dufresne identity.

\subsection*{{Operator limits}}

While structurally similar we state the limiting operator characterizations separately for $r=1$ and $r>1$. In addition to the restriction of the multispike limit operator to $\beta =1,2,$ or $4$, the $r=1$ operator has a secondary diffusion description not available for $r>1$.

We denote the $r$-spiked Wishart ensemble by $W_{n, \Sigma_r}$, or simply  $W_{n, \sigma}$ when $r=1$ (in which case it has the representation $B_{\sigma}^{} B_{\sigma}^T$ with $B_{\sigma}$ as in \eqref{B} and $m = n+a$).

\begin{theorem}  
\label{thm1}
Given a standard Brownian motion $ x \mapsto b(x)$, define: for any $a >-1$ and $\beta > 0$, 
\begin{equation}
\label{speedscale}
    \mm(dx) = e^{-(a+1) x - \stb b(x) }  dx,  \quad {\ss}(dx) = e^{ a x   + \stb b(x) }  dx, 
\end{equation}
along with the (random) integral operator
\begin{equation}
\label{IntegralForm}
 ( \G  f  )(x)  = \int_0^{\infty}  \int_0^{x \wedge y} \ss(dz)   f(y) \mm(dy) + \frac{1}{c} \int_0^{\infty} f(y) \mm(dy).
\end{equation}
For $c \in  (0, \infty]$, $\G =  \G_{\beta, a, c}$ is almost surely compact on $L^2 (\RR_+, \mm)$ with inverse eigenvalues $0 < \Lambda_0 < \Lambda_1 < \cdots$
corresponding to the problem $f  = \Lambda   \G  f   $.
 Further, as $n \rightarrow \infty$ with
$n \sigma \rightarrow c$, the ordered eigenvalues of the scaled one-spiked matrix ensemble 
  $n W_{n, \sigma}$ converge to the collection $\{ \Lambda_k \}$ in the sense of finite dimensional distributions.
  \end{theorem} 

What is attractive is that $\G$ is an explicit rank one perturbation of the hard-edge limit operator found in \cite{RR} for the unspiked  case $-$ which is recovered by setting $c = + \infty$ in \eqref{IntegralForm}. Just as in that case, pretending that everything in sight can be differentiated at will, $\G$ is seen to represent the Green's operator for the generator\footnote{In \cite{RR} the generator is denoted by $-\mathfrak{G}$, with $\mathfrak{G}^{-1}$ reserved for the Green's operator. This should't cause confusion here.}  
\begin{equation}
\label{differentialform}
   \mathfrak{L} =  e^{x}  \left( \frac{d^2}{dx^2} - (a + \frac{2}{\sqrt{\beta}} b'(x) ) \frac{d}{dx} \right),
\end{equation}
or $\mathfrak{L} = - \mathfrak{G}^{-1}$.
This can be made precise by recognizing $\mm(dx)$ and $\int^x \ss(dz)$ as the (random) speed measure and scale function for the associated diffusion $t \mapsto X_t$.   In particular, $\mathfrak{L} = \frac{d}{d \mm} \frac{d}{d \ss}$.

Fully specifying $\mathfrak{L}$ requires assigning a boundary condition at the origin.
It is easy to see that any  solution $f \in L^2[\mm]$ of $f(x) = \lambda \G f(x)$ is actually $C^{3/2-}$. Taking one derivative throughout this eigenvalue/eigenfunction identity produces
$$
 f'(0) = \lambda \int_0^{\infty} f(y) \mm(dy) 
$$
upon evaluating both sides at zero. And as we also have that $ f(0) = \frac{\lambda}{c} \int_0^{\infty} f(y) \mm(dy)$, we conclude that
\begin{equation}
\label{1dboundary}
   f'(0) =  c f(0),
\end{equation}
for any $L^2$ eigenfunction.
That is to say, just as in the case soft-edge spiking there is a Robin boundary condition at the origin of the (differential) operator limit.

From a probabilistic point of view  this boundary condition entails adjoining a killing measure ${c} \delta_0(x)$ to the speed-scale description.  Pathwise, $X_t$ is constructed from the more familiar process $ t \mapsto X_t^o$ with the same speed and scale, but with simple reflection at the origin.   With $L_t$ the local time of $X_t^o$ at the origin, $X_t$ is defined to agree with those of $X_t^o$ up to the killing time $\tau$ defined by the conditional distributions
$
           \PP ( \tau > t \, | \, X_{\cdot}^o)  = e^{ - c  L_t},
$ 
at which point the path is absorbed.  From this standpoint $c = +\infty$ denotes instantaneous killing, which is another way to recognize that the standard hard edge diffusion is absorbed on its first passage to the origin. In all cases, the behavior at infinity can be read off from the speed and scale and
depends on whether $a \ge 0$ or $ a < 0$.  For $a \ge 0$, $+\infty$ is an entrance, but not exit, point for the process.  For $a \in (-1, 0)$,
$+\infty$ is both entrance and exit.  In the latter case, it is appropriate to view the process as reflected at infinity.   Detailed background on this sort of construction and boundary classification can be found in Chapter 4 of \cite{IM}.

Moving to $r > 1$, the limit is still a random integral operator, but now with an $r \times r$ matrix kernel, acting on the appropriate $L^2$ space of vector valued functions. Again, as we use the block tridiagonalization procedure of Bloemendal-Vir\'ag, the operator itself has no natural extension off of  $\beta=1, 2,$ or $4$.

 Denote by $\mathbb{F}$ the field of real, complex, or quaternion numbers. By the corresponding $\beta=1, 2,$ or $4$ Brownian motion we mean the $r \times r $ matrix of independent unit $\mathbb{F}$-valued Brownian motions at all off-diagonals entries, and  independent real Brownian motions with diffusion coefficient $\frac{1}{\beta}$ along the diagonal. 
 

\begin{theorem}
\label{thm2}
Take $r > 1$ and $\beta =1, 2$, or $4$.  Define the $r \times r$ matrix processes:  
\begin{equation}
\label{Adefinition}
    d \A_x = 
                   \A_x dB_x + ( - \frac{a}{2 } + \frac{1}{2\beta  } ) \A_x dx,   \quad  \A_0 = I,
\end{equation}
with the appropriate $\beta = 1,2, 4$ matrix Brownian motion ${x \mapsto B_x}$.  Set 
\begin{equation}
\label{MandS}
\mathcal{{M}}_x =   e^{-rx} \mathcal{A}_x^{\phantom{l}} \mathcal{A}_x^{\dagger}, \quad
  \mathcal{{S}}_x  =     (  \A_x^{\phantom{l}} \A_x^{\dagger} )^{-1}
\end{equation}
Then, if $(n \Sigma_r)^{-1}$ converges in norm to a non-negative definite  $C^{-1}$, the limiting  $r$-spiked hard edge operator  $\mathfrak{G}_{r} =   \mathfrak{G}_{r, \beta, a, C}$
reads 
\begin{align}
\label{rIntegralForm}
  \mathfrak{G}_{r} f(x)  =  \int_0^{\infty} \left(  \int_0^{x \wedge  y} \mathcal{{S}}_z dz \right)    \mathcal{{M}}_y f(y) dy 
     +    {C}^{-1} \int_0^\infty  \mathcal{{M}}_y f(y) dy,
\end{align}
and is positive compact on $L^2 [\mathcal{M}] = \{ f \in \mathbb{F}_r: \int_0^{\infty} f^{\dagger}(x) \mathcal{M}_x f(x) dx < \infty \}$ to itself. In particular, the ordered eigenvalues of $n W_{n, \Sigma}$ converge (again in the sense of finite dimensional distribution) to the $L^2$-spectrum $\Lambda_{r,0} < \Lambda_{r,1} < \cdots $ for the problem $f = \Lambda \mathfrak{G}_r f$.
\end{theorem}

Putting $r=1$ in the equations \eqref{Adefinition} and \eqref{MandS} which define the matrix kernel, both $\mathcal{M}_x$ and  $\mathcal{S}_x$ reduce to the speed and scale $(\mm(dx) /dx, \ss(dx) /dx)$ functions from \eqref{speedscale}, as they must. 
Further, the eigenfunctions of $\mathfrak{G}_r$ will again be $C^{3/2-}$ entry-wise, and the inherited boundary condition at zero is
now $f(0) = C f'(0)$ in complete analogy to \eqref{1dboundary}.
On the other hand, for $r >1$ there is no apparent interpretation of $\mathfrak{G}_r^{-1}$ as a diffusion generator.

Note that the statement allows for some number eigenvalues of $C^{-1}$ to vanish. This translates into a Dirichlet condition at the origin for the corresponding coordinates of $f$.

Theorems \ref{thm1} and \ref{thm2} are proved in Sections \ref{section:1spike} and \ref{section:multispike}, respectively.
The proof of Theorem \ref{thm1} is a re-working of the main result of \cite{RR}. By virtue of having to deal with matrix rather than scalar kernels the proof of Theorem \ref{thm2} is  more elaborate, but  the overall strategy remains the same.
 In each case the 
matrix $(n W_n)^{-1}$, with $W_n = W_{n,\sigma}$ or $W_{n, \Sigma}$,  is embedded as an operator  into $L^2[0,1]$ (of functions taking values in 
$\mathbb{R}$ or $\mathbb{F}^r$).   After this embedding, the proofs demonstrate a coupling for which $(n W_n)^{-1}$ converges to an integral operator ($G$ or $G_r$, defined below in \eqref{eq:LimitingOperator1} or \eqref{eq:LimitingOperator2}, respectively) in strong operator norm. Thus, we have convergence of any finite number of eigenvalues and their corresponding eigenvalues (as elements of $L^2[0,1]$) to the eigenvalues/eigenvectors of these ($r$ dependent) $G$ operators. The half-line operators  $\mathfrak{G}$ and $\mathfrak{G}_r$ are related to the $G$ and 
$G_r$  by an explicit change of variables.

\subsection*{Riccati correspondence and PDEs}

The classical Riccati substitution was used in \cite{RRV} to provide a description of  the Tracy-Widom($\beta$) laws in terms of the explosion times of a certain diffusion process.   Applying the same ideas to the spiked soft edge operators, \cite{BV1, BV2} concluded that the space-time generator of that diffusion  yields a completely analytic characterization of the family of spiked/deformed soft edge laws. In brief, the introduction of the variable boundary condition produces the previously missing space-like variable in the corresponding PDE(s).  These PDEs have subsequently been used to confirm a conjectured Painlev\'e formula for the spiked $\beta=4$ Tracy-Widom law form \cite{Wang}, and more recently has led to  Painlev\'e formulas at $\beta =6$ \cite{Rumanov1} (the first appearance of such outside of $\beta=1,2,4$).

Working with the differential forms of $\mathfrak{G}$ or $\mathfrak{G}_r$ yields a similar picture at the spiked hard edge. The rank one  case follows directly from the Riccati diffusion for the ``standard" beta hard edge introduced in \cite{RR}.

\medskip

\noindent
{\bf{Rank one hard edge process}} Define $\PP_{\mu, c}$ to be the measure on paths $q: [\mu, \infty) \mapsto [-\infty, +\infty]$ induced by the (strong) solution of
\begin{equation}
\label{r1q}
    dq_x   =  \frac{2}{\sqrt{\beta}} q_x db_x + \left( (a + \frac{2}{\beta} ) q_x - q^2_x - e^{-x} \right) dx, 
\end{equation}
with initial condition $q_{\mu} =c $ and where $q$ is restarted at $+\infty$  after any explosion to $-\infty$. In particular, the exit and entrance points 
$-\infty$ and $+  \infty$ are joined so that the $q$ paths may be viewed as continuous for all time.

\medskip

The hard-edge($\beta$) laws introduced in \cite{RR}  correspond  to starting  $q$ from the entrance barrier $c=+\infty$.  Note also that the Riccati process there puts the spectral parameter inside the drift: the $e^{-x}$ term in \eqref{r1q} is replaced with $\lambda e^{-x}$. Here we are simply making the change of variable $\lambda = e^{-\mu}$ in order to move this variable into a starting ``time." In any case, it is clear that once $q$ hits zero at any finite time $x$, it immediately becomes and stays negative,
at which point there is a positive probability of exploding to $-\infty$. The spectral counting function is then given by counting the passes to zero of the pieced together process, restarted at $+\infty$ after any explosion.

\begin{theorem} 
\label{r1Riccati}
Let $0 < \Lambda_0 < \Lambda_1 < \cdots$ be the limiting spiked hard edge eigenvalues, the inverse eigenvalues of $\mathfrak{G} = \mathfrak{G}_{\beta, a, c}$, for any $\beta >0, a > -1$ and $c \in (0, +\infty]$. Notate $F_k(\mu, c) = \pr(\Lambda_k > e^{-\mu})$, with $F = F_0$ for simplicity.  Then,
\begin{equation}
 \label{r1F}
      F_k( \mu, c) = \PP_{\mu, c} \left( q_x \mbox{ vanishes at most } k \mbox { times} \right).
 \end{equation}
It follows that $F$ is the unique bounded positive solution of
\begin{equation}
\label{r1PDE}
    \frac{\partial F}{\partial \mu}   +  \frac{2}{\beta} c^2     \frac{\partial^2 F }{\partial c^2}   + ( (a +\frac{2}{\beta}) c - c^2   - e^{-\mu}) 
      \frac{\partial F}{\partial c}  =0,
\end{equation} 
subject to
\begin{equation}
\label{BCs}
   \lim_{\mu \uparrow \infty} F( \mu, c) = 1 \mbox{ for all } c > 0 , \quad  \lim_{c \downarrow 0}  F(\mu, c) = 0 \mbox{ for all } \mu < \infty.
\end{equation}
The higher order distributions $F_k$ solve the same PDE, but with the nested boundary conditions $\lim_{c \downarrow -\infty} F_k(\mu,c) = F_{k-1}(\mu, +\infty)$ in place of $\lim_{c \downarrow 0} F(\mu,c) =0$.
\end{theorem}

As a spiked parameter $c$ is to be taken positive, but $q$ can obviously be started from any  $c \in \RR$. In addition, when counting  $k \ge 1$ vanishings it is natural to consider starting points $c \le 0$, since subsequent vanishes can only occur after an explosion to $-\infty$. In this way the definition of $F_k$ is extended to negative $c$ by the right hand side of \eqref{r1F} in a consistent way. For example, $F(\mu, c) := 0$ for $c \le 0$. 

The PDE characterization of $F$ (not spelled out in \cite{RR}) is more or less  immediate from the passage time description \eqref{r1F}.  Rumanov \cite{Rumanov} has already used \eqref{r1PDE} to find formulas for $F$ at $\beta=2,4$ in terms of Painlev\'e III.  That same paper provides new proofs of the $\beta=2,4$ spiked soft edge laws using the analog of the PDE \eqref{r1PDE} derived in \cite{BV1}. It is further worth mentioning that more recently Rumanov \cite{Rumanov1} has found a Painlev\'e II formula at the $\beta = 6$ soft edge (the first known outside $\beta=1,2,4$) once more  using the spiked/PDE picture of \cite{BV1}.

For the multi-spiked case, we introduce the following.

\medskip

\noindent
{\bf{Rank $r>1$ hard edge process }} Define $\PP_{\mu, (c_1, c_2, \dots c_r)}$ to be the measure on (non-intersecting) paths ${\mathbf{q}}: [\mu, \infty) \mapsto [-\infty, \infty]^r$ governed by
\begin{equation}
\label{rrq}
      d q_{i,x} =  \stb q_{i,x} db_i +  \left(  (a + \frac{2}{\beta}) q_{i,x} - q_{i,x}^2 - e^{-r x} +  q_{i,x}  \sum_{j \neq i}  \frac{ q_{i,x} + q_{j,x}}{ q_{i,x}- q_{j,x} }  \right) dx,
\end{equation}
with initial condition $q_{i, \mu} = c_i $, $i=1,\dots, r$ (and convention $c_1\ge c_2 \ge \cdots$).  The explosion/restart now pertains to the lowest 
 particle, and  $\mathbf{q}$ is made continuous on the Weyl chamber for all time by following any  such restart by a cyclic relabeling of the indices.

\medskip

\begin{theorem} 
\label{thm:MultiRiccati}
Now let  $\beta=1,2,$ or $4$ and denote by $0 < \Lambda_{r,0} < \Lambda_{r,1} < \cdots$  the inverse eigenvalues of $\mathfrak{G}_r = \mathfrak{G}_{r, \beta, a, C} $ with $a > -1$ and $C$ itself possessing eigenvalues $c_1, \dots, c_r  \in (0,\infty]^r$.  Then
$F_k(\mu, c_1, c_2, \dots c_r) = \pr(\Lambda_{r,k} > e^{-\mu})$ has the representation: 
$$
  F_{k}(\mu, c_1, \dots, c_r) = \PP_{{\mu}/{r}, (c_1 , \dots, c_r)} \left(\mbox{at most } k \mbox{ zeros among the coordinates of } \mathbf{q}_x  \right).
$$
And so, with  $F = F_0$ and $\psi(c) = (a + \frac{2}{\beta}) c - c^2 - e^{-r \mu} $,
\begin{equation}
\label{rrPDE}
\frac{\partial F }{ \partial \mu}   + \sum_{i=1}^r \left( \frac{2}{\beta}  c_i^2  
\frac{\partial^2 F}{\partial c_i^2}    + \psi(c)  \frac{\partial F}{\partial c_i} \right)
+
   \sum_{i < j}   \left( \frac{c_{i}+ c_j}{ c_{i} - c_{j}} \right) (c_i  \frac{\partial F}{\partial c_i}    -  c_j  \frac{\partial F}{\partial c_j} ) = 0,      
\end{equation}
subject to the appropriate boundary conditions analogous to \eqref{BCs}. The solution is unique up to permutations of the variables
$c_1, c_2, \dots, c_r$. The higher order distributions solve the same PDE, linked together inductively as in $\lim_{c_r \downarrow -\infty} F_k(\mu, c_1, \dots, c_r)$ $ = $ $ F_{k-1}(\mu, +\infty, c_1, \dots, c_{r-1})$.
 \end{theorem}

As indicated above, the value $a=0$ is critical for the $r=1$ operator in terms of the interpretation of the boundary condition at infinity. For $a \in (-1, 0)$ the condition is Neumann, and for $a \ge 0$ one can view the condition as Dirichlet (the diffusion generated by $\mathfrak{G}^{-1}$ has a natural boundary condition there). In the Riccati picture, the difference between Neumann and Dirichlet is whether eigenvalues are counted by passages to $0$ or to $-\infty$. One might expect the same dichotomy at $a=0$ in the case $r >1$. What we can show is the following.

\begin{corollary}
\label{cor:Large_a} 
 In either Theorem \ref{r1Riccati} or  Theorem \ref{thm:MultiRiccati}, if $a \ge r-1$ we can replace the characterization of the $k$-th eigenvalue distribution function in terms of the
probability of (at most) $k$ zeros with that of (at most) $k$ passages to $-\infty$.  
\end{corollary}

Theorems \ref{r1Riccati} and \ref{thm:MultiRiccati} along with the corollary are proved in Section \ref{section:Riccati}.

\subsection*{The hard-to-soft transition}

In \cite{BV1, BV2} the outcome of the Riccati map is that the one-spiked soft edge distributions are described in terms of the probability that the process
\begin{equation}
\label{softdiffusion1}
  d p_x = \stb db_x + (x - p_x^2) dx,
\end{equation}
   begun at $w$,  
has a given number of explosions to $-\infty$ (after subsequent restarts at $+\infty$).  In particular, with $\lambda_0(\beta, w) < \lambda_2 (\beta, w) < \cdots $ the eigenvalues of the ``deformed" Stochastic Airy Operator, in \cite{BV1} it is proved that
$$
   \pr (\lambda_k(\beta, w) \le \lambda) = \PP_{\lambda, w} ( p_x \mbox{ explodes at most $k$ times}).
$$
Here $\PP_{\lambda, w}$ indicates that $p$ is started from the time/space point $(\lambda, w)$, as should be compared with Theorem \ref{r1Riccati}. Again recall that $-\lambda_0(\beta, +\infty)$  has the Tracy-Widom($\beta$) distribution \cite{RRV}.
 
 Likewise, for $r >1$, the process \eqref{softdiffusion1} is replaced by the joint diffusion $\mathbf{p} = (p_1, \dots,  p_r)$:
 \begin{equation}
 \label{softdiffusion2}
   dp_{i, x} =   \stb db_{i,x} +  \left(rx - p_{i,x}^2 + \sum_{k \neq i} \frac{2}{p_{i,x} -p_{k,x}} \right) dx.
 \end{equation}
 And with now $ \{ \lambda_{r,k}(\beta, \mathbf{w} )\}$, $\mathbf{w} = (w_1, \dots, w_r)$ the eigenvalues of the operator \eqref{SAOr}, among the results of \cite{BV2} is the fact that: for $\beta =1,2,4$,
 $$
    \pr (\lambda_{r,k}(\beta, \mathbf{w}) \le \lambda) = \PP_{\lambda/r, (w_1, \dots, w_r)} ( \mathbf{p} \mbox{ explodes at most $k$ times}).
 $$
 Here again, the process at started from $p_i = w_i$, $i=1, \dots r$ (with the convention that $w_r \le \dots \le w_1$) at the time-like parameter $x = \lambda/r$.  Compare  Theorem \ref{thm:MultiRiccati}.
 
 Back in \cite{RR} we proved that, for all $\beta$, the rescaled (non-spiked) minimal hard-edge eigenvalue converges in distribution to the general Tracy-Widom law as $a \rightarrow \infty$. This was accomplished by showing that the law of passage time to $-\infty$ for a suitably rescaled hard-edge process \eqref{r1q}, begun at $+\infty$, converges to that for the soft-edge process \eqref{softdiffusion1}, also begun at $+\infty$.  Previously, this ``hard-to-soft transition" had been verified for $\beta=1,2,4$ using the explicit Bessel/Airy correlation functions \cite{BF}. In all cases, the intuition is that $a$ is a dimension-difference parameter in the classical Wishart ensemble. Exchanging limits freely, $a \rightarrow \infty$ would correspond to an asymptotically rectangular Wishart, for which the smallest eigenvalue pulls away from the origin and has ``regular" Tracy-Widom fluctuations.
 
Given the above,  it should be no surprise that we can re-capture all the various spiked soft edge laws by rescaling our spiked hard-edge laws. Here we prove the convergence in the sense of finite dimensional distributions of the hard and soft edge point processes (rather than just for the first point as was done in \cite{RR}).


\begin{theorem} 
\label{thm:HardSoft}
Again denote by $\Lambda_0, \Lambda_1, \dots$ the (potentially one-spiked) hard edge eigenvalues. For $a \rightarrow \infty$ and  $c = c(a) \rightarrow \infty$ with $\lim_{a \rightarrow \infty}  a^{-2/3}({c}(a) - a) = w \in (-\infty, \infty]$, there is the convergence in law,
$$
        \left\{  \frac{a^2 - \Lambda_i(\beta, 2a , {c}(a) )} {a^{4/3}} \right\}_{i=0, \dots,k} 
        \Rightarrow \{ \lambda_i(\beta, w) \}_{i=0, \dots, k},        
$$
for any finite $k$. Note the shift $a \mapsto 2a $ in the ``dimension discrepancy" parameter in the $\Lambda_i$.

The same appraisal holds for the $r$-spiked laws. Any finite collection of points from the family $\{a^{2/3} -  a^{-4/3}\Lambda_{r,i}(\beta, 2a, \mathbf{c}(a)) \}$ converges in law to the corresponding $\{ \lambda_{r,k}(\beta, \mathbf{w})\}$, assuming $(\mathbf{c}(a))_i =c_i(a)$ is such that
$\lim_{a \rightarrow \infty}  a^{-2/3}({c}_i(a) - a) = w_i \in (-\infty, \infty]$, $\mathbf{w} = (w_1, \dots, w_r)$.
\end{theorem}

\begin{remark} 
\label{rem:HtoS}
Taking the point of view that $\{\Lambda_{r, i}\}$
and $\{\lambda_{r,i} \}$ are defined via the diffusions \eqref{rrq} and \eqref{softdiffusion2}, that is, without reference to the limit operators, the hard-to-soft convergence in the $r$-spiked case takes place for all $\beta \ge1$. 
\end{remark}

The proof of Theorem \ref{thm:HardSoft} makes up Section \ref{sec:HardSoft}.

\subsection*{Supercritical scaling}

It is now well understood that, at the soft edge, if the spikes are too large (or ``supercritical") the fluctuations of $\lambda_{max}$ are effectively Gaussian.  More precisely, for $r=1$ there is a normal limit law for the largest eigenvalue, while for $r > 1$ the rescaled largest eigenvalue converges to that for an $r \times r$ copy of the corresponding invariant Gaussian ensemble. For $\beta=2$ this phenomena is included in the results of \cite{BBP}. It can also be derived by a perturbation argument in the deformed soft-edge operators.

While there is no critical point for the hard edge, it is clear that the analog of supercritical spiking is, at least for $r=1$,  to consider the $c \rightarrow 0$ limit of the operator $\G_{\beta,a, c}$ after the fact.  In the multi-spiked case we take the simplest version of this, setting $C = cI$ with $c \rightarrow 0$. For $r=1$ we find:

\begin{corollary} 
\label{supercrit}
As $c \rightarrow 0$,  
$  \frac{1}{c} \Lambda_0(\beta, a, c) \Rightarrow \frac{1}{\beta}\chi^2_{{\beta}(a+1)}  $
in distribution.
\end{corollary}

Of course the $a \rightarrow \infty$ scaling limit of a $\chi^2_{\beta (a+1)}$ is Gaussian, recovering the soft-edge supercritical limit law.
The proof is standard perturbation theory: as $c \rightarrow 0$, plainly  $c \, \G_{\beta, a,c}$  tends (with probability one) to the rank one operator defined by integration against $\mm(dx)$. That is, the ground state is the constant function with corresponding eigenvalue $\int_0^{\infty} \mm(dx)$, and the only question then is to understand that law.  But an identity due to Dufresne \cite{Duf} shows that
$$
    \int_0^{\infty}   e^{ 2( b_x -  \mu x) } dx  \sim \frac{1}{ 2 \gamma_{\mu}}, 
$$
for $\gamma_{\mu}$ a gamma random variable of the indicated parameter, which is exactly what is needed. 

For $r > 1$ an identical perturbation argument gives: as $c \rightarrow 0$, 
\begin{equation}
\label{multisupcrit}
  c^{-1}  \Lambda_{r,0}(\beta, a, cI) \Rightarrow \lambda_{min} \left(  \int_0^{\infty} \M_x dx \right)^{-1} , 
\end{equation}
with  $\M_x$ is defined in \eqref{MandS}.  Recognizing $ \frac{1}{\beta}\chi^2_{{\beta}(a+1)} $ as a $1\times 1$ Wishart law prompts 
a conjectured matrix Dufresne identity: $\mathrm{spec} ( \int_0^{\infty} \M_x dx)$ should have joint law corresponding to the $r \times r$ inverse  $\beta =1,2,$ or $4$ Wishart distribution.  
This has subsequently been proved \cite{RV}, and we can state the following.

\begin{corollary} (See Corollary 10 of \cite{RV})
\label{cor:multisupcrit1}
 As $c \rightarrow 0$, $  c^{-1}  \Lambda_{r,0}(\beta, a, cI_r)$ converges in distribution to the law of the minimal eigenvalue of the $(r, r+a)$ 
 real, complex, or quaternion Wishart ensemble.
\end{corollary}

For completeness, in Section \ref{sec:Appendix} we include a proof of Corollary \ref{cor:multisupcrit1} for $\beta=2$ and integer $a$ 
using the Fredholm determinant approach of \cite{BBP}.

\section{Single-spike operator limit}
\label{section:1spike}

After a conjugation by the diagonal matrix of alternating signs, the matrix model has the form $n W_{n, \sigma} = n B_{\sigma} B_{\sigma}^T$
in which the upper-bidigaonal $B_{\sigma}$ has entries
$$
   \sqrt{\sigma} \frac{\chi_{\beta(n+a)}}{\sqrt{\beta}}, \, \frac{\chi_{\beta(n-1+a)}}{\sqrt{\beta}}, \cdots, \mbox{ and } 
       - \frac{\chi_{\beta(n-1)}}{\sqrt{\beta}}, \, - \frac{\chi_{\beta(n-2)}}{\sqrt{\beta}}, \cdots, 
$$
along the main, and first off-diagonal, respectively.\footnote{In \cite{RR} we also flipped the indexing by conjugating as in 
$B \mapsto S B S^{-1}$ with $S_{ij} = (-1)^i \delta_{i+j - n-1}$. With the benefit of hindsight it is more convenient not to make this extra move.}
Keep in mind that, despite possibly repeated indices, all variables are independent.

Next, any matrix $A = a_{i,j} \in \RR^{n \times n}$ can be embedded 
into $L^2[0,1]$ without changing the spectrum by setting: for $x_i = i/n$ for $i = 0, 1, \dots, n$ and $f \in L^2[0,1]$,
$$
     (A f) (x)  :=  \sum_{j=1}^n   a_{i,j} n \int_{x_{j-1}}^{x_{j}}  f(x) dx, \quad  \mbox{when }  x_{i-1} \le x < x_{i}.
$$
Then, by the inversion formula for bidiagonal matrices, as an integral operator,
\begin{equation}
\label{r1Kernel1}
   (\sqrt{n} B_{\sigma})^{-1}(x,y) =  \left( \one_{x > x_1}(x)+\one_{x \le x_1}(x) \frac{1}{\sqrt{\sigma}} \right) k_n(x,y),
\end{equation}
with the discrete (upper triangular) kernel
\begin{equation}
\label{r1Kernel2}
   k_n(x,y) = \frac{  \sqrt{\beta n }}{\chi_{\beta(n+a - i)}} \prod_{k=i+1}^{j} \frac{\chi_{\beta(n-k)}}{\chi_{\beta(n+a-k)}}  \,
   \mathbf{1}_{\Gamma_{ij}},
\end{equation}
in which   
\begin{equation}
\label{r1Kernel3}
    \quad \Gamma_{ij} = \{ 0 \le  x \le y  \le 1:  x \in (x_{i-1}, x_i], \,   y \in (y_{j-1}, y_j]  \}.
\end{equation}
Note that on diagonal  (if $i=j$) the product in \eqref{r1Kernel2} is understood to equal one.

Putting \eqref{r1Kernel1} together with its transpose we can write out  the action of $(n W_{n,\sigma})^{-1}$
on an $f\in L^2[0,1]$:
\begin{align}
\label{fullr1operator}
 (n W_{\sigma})^{-1}f(x)  & = \, \int_{x_1}^{1} \int_0^1  k_n(y,x)  k_n(y,z)   f(z) \,dz \, dy  \ +  \ 
  \frac{1}{n\sigma} k_n(0,x) \int_0^1 k_n(0,y) f(y) \, dy  \\
  &   :=      K_n^T K_n f (x) +  \frac{1}{n}\left(\frac{1}{\sigma}- 1 \right) k_n(0,x) \langle k_n(0,\cdot), f\rangle.  \nonumber
\end{align}
Here $K_n$ is the integral operator with kernel $k_n$ (defined in \eqref{r1Kernel2}, \eqref{r1Kernel3}) and $\langle \cdot, \cdot \rangle$ is just the $L^2$ inner product.

The point is that \eqref{fullr1operator} represents $(n W_{n,\sigma})^{-1}$ as a rank one perturbation of $K_n^T K_n$. As the latter is exactly the non-spiked model (just set $\sigma =1$), the relevant convergence/compactness properties of these operators have already been studied in \cite{RR}. Everything we need  to take norm limits throughout the right hand side of \eqref{fullr1operator} is summarized in the following.

\begin{proposition} (Lemma 6 of \cite{RR}) 
\label{PropForThm1}
Let $K_{\beta,a}$ be the integral operator on $L^2[0,1]$ with kernel
\begin{equation}
\label{limitkernel}
   k_{\beta,a}(x,y) := (1-x)^{-\frac{1+a}{2}}  \exp{ \left[  \int_x^y \frac{db_z}{\sqrt{\beta (1-z)}} \right]} 
      (1-y)^{a/2} \, \one_{x< y}.
\end{equation}
Here $z \mapsto b_z$ is a standard Brownian motion. Then, for any sequence of the operators $K_n$, $n \rightarrow \infty$, there is a subsequence $n' \rightarrow \infty$ 
and suitable probability space on which $K_{n'} \rightarrow K_{\beta, a}$ almost surely  in Hilbert-Schmidt norm.
\end{proposition}

\begin{proof}[Proof of Theorem \ref{thm1}] Since it serves as a model for the multi-spike case, we first sketch the main steps behind
the above proposition. Here we drop the $(\beta,a)$ subscripts on $k$ and $K$.

It is straightforward to show that there is the convergence
\begin{equation}
\label{discrete_conv}
   \frac{  \sqrt{\beta n }}{\chi_{\beta(n+a - i)}}  \Rightarrow \frac{1}{\sqrt{1-x}},
   \qquad  \prod_{k=i+1}^{j} \frac{\chi_{\beta(n-k)}}{\chi_{\beta(n+a-k)}} \Rightarrow \left(\frac{1-y}{1-x} \right)^{a/2} e^{\int_x^y 
   \frac{db_z}{\sqrt{\beta (1-z)}} } 
\end{equation}
as processes. So  $k_n(x,y) \rightarrow k(x,y)$ pointwise on $[0,1]^2$. With a bit more effort it is shown in \cite{RR} that there
is the estimate
\begin{align}
\label{discrete_tightness}
   k_n(x,y) & \le  \frac{C_n}{\sqrt{1-x}}   \left(\frac{1-y}{1-x} \right)^{a/2}  e^{C_n  (\phi(x) + \phi(y))}  \one_{x< y}  \\
                 & := K_{C_n}(x,y) 
    \quad \mbox{ with } \ \phi(z) = \left(1+\log\frac{1}{1-z}\right)^{1/2+},   \nonumber
\end{align}
for $C_n$ a tight random constant.\footnote{\cite{RR} reports an exponent of $3/4$ in the bounding function 
$\phi$, but the proof there shows that this can be replaced by anything greater than $1/2$.}
This comes down to a quantitative version of the law of the iterated logarithm for the random sum $\sum_{nx \le k \le ny} \log  \frac{\chi_{\beta(n-k)}}{\chi_{\beta(n+a-k)}}$.  And  then by the law of the iterated logarithm for Brownian motion, we also have that the limit kernel $k$ satisfies the same bound,
but with $C_n$ replaced by a fixed ({\em i.e.}, independent of $n$) random constant $C< \infty$

Now we can choose a subsequence $n' \rightarrow \infty$ over which $C_{n'}$ converges, and then realize this convergence as almost sure convergence on some probability space. With this set-up we have both $k_{n'}(x,y) \le K_C(x,y)$ and $k(x,y) \le K_C(x,y)$ after possibly adjusting $C$. But $ \int_{[0,1]^2} K_C^2(x,y) dx dy < \infty$, and this supplies the domination needed to conclude that 
$\| k_n - k \|_{L^2[0,1]^2} \rightarrow 0$.

The new object appearing in the spiked operator \eqref{fullr1operator} is the projection onto the function $k_n(0,x)$. But the estimates just employed show this also converges almost surely (in the same sub-sequential coupling) to the projection onto $k(0,x)$. Hence, 
if  $\sigma n \rightarrow c$, we have that $(n' W_{n', \sigma})^{-1}$  converges in norm to $G$ defined by
\begin{align}
\label{eq:LimitingOperator1}
G f (x)  = &   (K^T K f)(x)  + \frac{1}{c} \, k(0,x) \langle k(0,\cdot), f\rangle  \\
   = & \int_0^1 (1-x)^{a/2} e^{- \int_0^x  \frac{db_t}{\sqrt{ \beta (1-t)}}} \left(
                                              \int_0^{x \wedge y}  \frac{  e^{ 2 \int_0^z  \frac{db_t}{\sqrt{ \beta (1-t)}}}}{(1-z)^{a+1}} dz \right)
                                               (1-y)^{a/2} e^{- \int_0^y  \frac{db_t}{\sqrt{ \beta (1-t)}}} f(y)  dy  \nonumber \\
       &     +   \frac{1}{c}  (1-x)^{a/2} e^{- \int_0^x  \frac{db_t}{\sqrt{ \beta (1-t)}}} \int_0^{\infty}    (1-y)^{a/2} e^{- \int_0^y  \frac{db_t}{\sqrt{ \beta (1-t)}}} f(y)  dy, \nonumber                               
\end{align}
yet again on the chosen probability space. Norm convergence implies that we have convergence of finite parts of the spectrum
(this is standard, but Theorem 1 of \cite{RR} includes a proof in the present context). 
Finally, since any such extracted subsequence produces the same norm limit, we have that any finite collection of eigenvalues and eigenvectors of 
$n W_{n,\sigma}$ (the latter as elements of $L^2[0,1]$) converge jointly in law over the full sequence $n \rightarrow \infty$.

Just to indicate the map from $G$ to $\mathfrak{G}$, one checks by Brownian time change that
$$
    k(0,1-e^{-x}) e^{-x/2} = e^{ - \frac{1}{\sqrt{\beta}} b_x  - \frac{a+1}{2} x}  =   \sqrt{\frac{\mm(dx)}{dx }},
$$
where the equalities are in law and $b_x$ is a new Brownian motion. After the same change of variables, the kernel of 
$K^T K$ satisfies
$$
 ( k^T k)(1-e^{-x}, 1-e^{-y}) e^{-x/2} e^{-y/2} = \int_0^{x \wedge y} \ss(dz) \,  \sqrt{\frac{\mm(dx)}{dx } } \sqrt{\frac{\mm(dy)}{dy }}.
$$
Together this means that $G$  acting on functions $f \in  L^2[0,1]$ possesses the same eigenvalues as  
$\G$ acting functions  on $\mathfrak{f}(x) = $ $ f(1-e^{-x}) e^{ \frac{a}{2}x + \frac{1}{\sqrt{\beta}} b_x} \in L^2[ [0,\infty), \mathfrak{m}]$.
\end{proof}

\section{Multi-spiked operator limit}
\label{section:multispike}

The starting point is the block bidiagonalization procedure of 
Bloemendal-Vir\'ag, introduced in \cite{BV2} to analyze multiple spikes at the soft edge.  
Let  $\Sigma = \Sigma_r \oplus I_{n-r}$ with  $\Sigma_r$ an $r \times r$ positive definite matrix, 
and consider  $ X \Sigma^{1/2} $  with $X$ an $n \times (n+a)$ Gaussian matrix  with independent unit real, complex, or quaternion entries.
 The key ingredient from \cite{BV2} is the existence  of  $U$, $V$ with $UU^{\dagger}= VV^{\dagger} = I$  so that $ U X  V$ is ``upper block bidiagonal" with $r \times r$ blocks. Denoting by $D_k$, $k=1, \dots \lceil n/r \rceil $ and $O_k$, $k = 1, \dots, \lceil n/r \rceil-1$ the diagonal and immediately above diagonal blocks, it holds that
\begin{equation}
\label{ondiag}
     (D_k)_{ij} = \left\{ \begin{array}{ll}  \frac{1}{\sqrt{\beta}} \chi_{\beta (n + a -r(k-1) -i+1)} & \mbox{if } i=j, \\
                                                                    g_{ij} & \mbox {if } j> i,  \\
                                                                      0        & \mbox{otherwise},  \end{array} \right.
\end{equation}
and
\begin{equation}
\label{offdiag}
     (O_k)_{ij} = \left\{ \begin{array}{ll}  \frac{1}{\sqrt{\beta}} \chi_{\beta (n-rk -i +1)} & \mbox{if } i=j, \\
                                                                    g_{ij} & \mbox {if } j< i,  \\
                                                                      0        & \mbox{otherwise}. \end{array} \right.
\end{equation}
The $\chi's$ are again $\chi$ random variables of the indicated parameter, but now the $g$'s are $\mathbb{F}$-valued unit Gaussians (which prevents a natural extension to general beta). 
 As usual, despite repeated notation, all random variables in sight are independent. Note also that the final blocks may need to be truncated in the obvious way (whenever $r$ fails to divide $n$).  
 
 What is important is that the outcome of applying the above procedure to $ X \Sigma^{1/2}$ rather than to just $X$ is simply to replace $D_1$ with $ D_1 \Sigma_r^{1/2}$.  This defines a block bidiagonal random matrix $B = B_{\Sigma^{1/2}}$  such that $W_{n,\Sigma} = B B^{\dagger}$ has the same eigenvalues as  $X \Sigma X^{\dagger}$.

\subsection{Identifying the limit}

The program follows the $r=1$. The scaled matrix $(\sqrt{n} B_{\Sigma^{1/2}})^{-1}$ is imbedded into 
the appropriate $L^2$ space.  The block structure dictates that this should be the vector valued functions $f$: $[0,1] \mapsto 
\mathbb{F}_r$  with square-norm $\| f \|^2 = \int_{[0,1]}   f(x)^{\dagger} f(x)  dx < \infty$.

One can again use the (block) bi-diagonal structure to compute the inverse explicitly. In conjunction with a conjugation by a matrix of alternating signed $r\times r$ identity matrices (which has the effect of replacing each $O_k$ with $-O_k$), the action of   $(\sqrt{n} B_{\Sigma^{1/2}})^{-1}$ can be identified with that of the discrete matrix kernel:
\begin{equation}
\label{discreterkernel}
   k_{n,r}(x,y) =  \frac{\sqrt{n}}{r}  {D}_i^{-1}   \left( \prod_{k=i+1}^j O_{k} D_k^{-1} \right) \,  \mathbf{1}_{i\le j}, \quad
        \mbox{ for } x \in [x_{i-1}, x_i), \, y \in [y_{j-1}, y_j).
\end{equation}
Here $x_i = \frac{ri}{n}$, $y_j =\frac{rj}{n}$ and $j \le [ n/r ] $.\footnote{Throughout  $[\cdot]$ denotes the (appropriate) integer part.}
 The point being that the discretization scale is now $\triangle = \frac{r}{n}$, which accounts for the $\frac{1}{r}$ prefactor on \eqref{discreterkernel}.
This object left multiplies functions $f(y) \in \mathbb{F}_r$ (prior to integration in the $y$ variable). Compare the $r=1$ kernel \eqref{r1Kernel2}.

Next, setting 
\begin{equation}
{A}_n(x) = \prod_{k= 1}^{  [ nx/r ]}   O_k D_{k+1}^{-1},
\label{eq:An}
\end{equation} 
we can write the kernel \eqref{discreterkernel} in a streamlined fashion, 
\begin{equation}
\label{discreterkernel2}
  k_{n,r}(x,y) =  \frac{ \sqrt{ n}}{r} {D}_{ [ \frac{n}{r} x ] }^{-1} {A}_n(x)^{-1} {A}_n(y)  \mathbf{1}_{0\le x<y\le1},
\end{equation}
granted some ambiguity in the indexing. This matrix process $x \mapsto A_n(x)$ satisfies the discrete equation
\begin{equation}
\label{discreteAnEq0}
   A_n( x + \Delta) - A_n(x) = A_n(x)  \left( {O}_{[ n x / r ] +1} {D}_{ [ n x / r ] + 2}^{-1}   - {I} \right),
\end{equation}
with the convention that $A_n(0) = I$, the $r\times r$ identity. Estimating moments of the entries of the $(O D^{-1} - I)$
multiplier, it is not difficult to convince oneself  that $x \mapsto A_n(x), x \in [0,1]$ converges as a process to
 $x \mapsto {A}_x$ defined as the (unique) solution of the It\^o equation:
\begin{equation}
\label{Aeq}
    d A_x =  \frac{1}{\sqrt{r (1-x)}} A_x dB_x +
    \left( -  \frac{a}{2 r(1-x)} + \frac{1}{2\beta r(1-x) } \right) A_x dx, \quad A_0 = I.
\end{equation}
Here $B_x$ is an $r \times r$ matrix Brownian motion introduced above. Each diagonal is a real Brownian motion with diffusion coefficient $\frac{1}{\beta}$, each
off-diagonal is a unit $\mathbb{F}$-valued Brownian motions, and all coordinate processes are independent. This structure is simply a consequence of the blocks $(D_k, O_k)$ having (real) $\chi$-variables along the diagonal and $\mathbb{F}$-valued Gaussians otherwise.

Moving along, it is easier still to prove the process level convergence $\frac{1}{\sqrt{n}} {D}_{ \lceil \frac{n}{r} x \rceil } \Rightarrow \sqrt{1-x} I$. The (preliminary) conclusion would be that $k_{n,r}$ converges in law to
\begin{equation}
\label{limitrkernel}
   k_r(x,y) =  \frac{1}{r \sqrt{1-x}} {A}_x^{-1} {A}_y^{\phantom{1}}   \mathbf{1}_{0\le x<y\le1},
\end{equation}
in the uniform-on-compacts topology of $[0,1) \times [0,1)$. Hence, granting the existence of a positive definite
${C} = \lim_{n \rightarrow \infty} n \Sigma_r$, the candidate operator limit of $(n W_{\sigma})^{-1}$ is:
\begin{align}
\label{eq:LimitingOperator2}
G_r f(x)   = & \frac{1}{r^2}   \int_0^1 \int_0^{x \wedge  y} {A}_x^{\dagger}  \left( \frac{1}{1-z} {A}_z^{-\dagger} {A}_z^{-1}  \right) {A}_y f(y) dz dy  + \frac{1}{r}
   {A}^{\dagger}_x  {C}^{-1}   \int_0^1 {A}_y f(y) dy,
\end{align}
again acting on $f \in L^2([0,1] \mapsto \mathbb{F}_r)$. This is the analog of \eqref{eq:LimitingOperator1}. 

Finally, to go from
$G_r$ to the representation of $\mathfrak{G}_r$ in the statement of Theorem \ref{thm2} is through a similar change of variables employed at $r=1$:
now $1-x \mapsto e^{-rx}$. In particular, defining $\A_x = A_{1-e^{-rx}},$ $x \in [0, \infty)$ we have that
\begin{equation}
\label{NAeq}
    d \A_x =  
                   \A_x dB_x + ( - \frac{a}{2 } + \frac{1}{2\beta  } ) \A_x dx, \quad \A_0 = I,
\end{equation}
with a new matrix $\mathbb{F}$-Brownian motion $B_x$, compare  \eqref{Adefinition}. Then, similar to the one-dimensional case, after this substitution one may symmetrize with respect to $L^2[\mathcal{{M}}]$. Recall that $\mathcal{M} = e^{-rx} \A_x \A_x^{\dagger}$ and
 $\mathfrak{f} \in  L^2[\mathcal{{M}}]$ if  $\int_0^\infty ( \mathfrak{f}^\dagger \M \mathfrak{f})(x) dx < \infty$.  In particular, view the transformed
 \eqref{eq:LimitingOperator2} as acting of functions $\mathfrak{f}$ of the form
$ \A_x^{-\dagger} f(1-e^{-rx}) e^{rx/2}$ which reside in $L^2[\mathcal{{M}}]$ by construction.
This sends eigenvectors $f$ of $G_r$ into eigenvectors $\mathfrak{f}$ of $\mathfrak{G}_r$ with the corresponding eigenvalues unchanged.



\subsection{Outline of the proof}

The convergence of the spectrum requires some form of norm convergence  of the operators $K_{n,r}$ to $K_r$ (these denoting the integral operators associated with the kernels $k_{n,r}$ and $k_r$ of \eqref{discreterkernel2} and \eqref{limitrkernel}). As in the proof of Theorem  \ref{thm1} (recall the argument surrounding Proposition \ref{PropForThm1}) things follow readily given that we can show:

\begin{theorem}
\label{thm2crux} The exists a probability space on which all $k_{n,r}$  and $k_r$ are defined and, given any sequence $n \uparrow \infty$, a subsequence $n' \uparrow \infty$ such that 
\begin{equation}
    \int_0^1 \int_0^1 \| k_{n',r}(x,y) - k_r(x,y) \|^2 dx dy \rightarrow 0,
\end{equation}
with probability one. 
Here $\| \cdot \|$ is the (matrix) operator norm.
\end{theorem}

\medskip

{\em Step 1} is to prove  (Lemma \ref{lem:ContinuumHSBound} of Section \ref{sec:LimitOperator}) that the limit kernel $k_r(x,y) =  \frac{1}{r\sqrt{1-x}} A_x^{-1} A_y \one_{x \le y}  $ is almost surely Hilbert-Schmidt  (showing that $\mathfrak{G}_r$ has the properties claimed in Theorem \ref{thm2}).

\medskip

{\em Step 2} constructs a (random) dominating kernel, on a restricted domain. We show that, for large enough $c$ 
there is a kernel $k_{r,c}(x,y)$ such that: with a tight random sequence $C_n$, 
\begin{equation}
\label{step21}
     k_{n,r}(x,y) \le C_n k_{r,c}(x,y)  \mbox{ for } 0 \le x \le y \le 1-  c \frac{\log n}{n},
\end{equation}
where $k_{r,c}(x,y)$ satisfies
\begin{equation}
\label{step22}
     \int_0^1 \int_0^1 || k_{r,c}(x,y) ||^2 dx dy < \infty.  
 \end{equation}
See Corollary \ref{cor:DiscreteBound} of Section \ref{sec:drectattempt}. This step also relies heavily on Lemma \ref{lem:Euler} of Section \ref{sec:Euler} which represents the product $A_n(x)$ as a fairly explicit (though approximate) Euler scheme for the limiting $x \mapsto A_x$. Along the way we conclude that $k_{n,r}(x,y) \rightarrow k_r(x,y)$ pointwise.

\medskip

{\em Step 3} uses Step 2 as input to show the rest of the range of integration can be neglected:
\begin{equation}
\label{step3}
 \iint\limits_{ {0 \le x \le y \le 1},  \, { y \ge 1- c \frac{\log n}{n} } } \| k_{n,r}(x,y) \|^2 \, dy dx  \rightarrow 0
\end{equation}
 in probability (Lemma \ref{lem:ThrowOutTail} of Section \ref{sec:ThrowOutTail}). 
 
From here the proof follows the course of the $r=1$ case. The only real wrinkle  is the requirement of the cutoff (off the last $O(\log n)$ discrete time steps) in the domination \eqref{step21} (compare this to the bound \eqref{discrete_tightness} used at $r=1$). Tracking the growth of the matrix kernels is just more involved (we also make due with a suboptimal bounding kernel $k_{r,c}$ in terms of its behavior near $x,y = 1$, again compare \eqref{discrete_tightness}). 

Anyway, one is  now free to choose a subsequence and a probability space on which \eqref{step3} takes place almost surely and the bound \eqref{step21} holds almost surely with the tight sequence $C_n$ replaced by a deterministic $c$ (which bounds the chosen sub-sequencial limit of $C_n$). Together with Step 1 this provides the necessary domination for the conclusion of Theorem \ref{thm2crux}. (We have already mentioned we will have that $k_{n,r} \rightarrow k_r$ pointwise, and this can also be realized almost surely by choice of a further subsequence.)

\subsection{The limit operator}
\label{sec:LimitOperator}

As for  $r=1$, the limit kernel $k_r$ is Hilbert-Schmidt (and so $\mathfrak{G}_r$ is compact).
The advantage of the continuum setting is that one can readily derive an sde for the singular values of the family $(x,y) \mapsto A_x^{-1} A_y$. The 
 analysis is similar to that for the Lyapunov exponents for Brownian motion on the general linear group, see for example \cite{Norris}.

\begin{lemma}
\label{lem:ContinuumHSBound}
With probability one,
\begin{equation}
\label{HSr}
       \int_0^1 \int_x^1 \frac{1}{1-x} ||  A_x^{-1} A_y ||^2 dy dx = r^2 
        \int_0^{\infty} \int_s^{\infty} e^{-rt} || \A_s^{-1} \A_t^{} ||^2 dt ds < \infty,
\end{equation}
recall (\ref{Aeq}) and (\ref{NAeq}).  Again $|| \cdot ||$ is the matrix operator norm.
\end{lemma}

\begin{proof}  
As all norms are equivalent we can reduce this to an estimate on the largest singular value of the two-parameter family 
$  \A_s^{-1} \A_t$.
Thus,  introducing the processes $ t \mapsto U(t; s) = \A_t^{\dagger} \A_s^{-\dagger} \A_s^{-1}  \A_t $ 
subject to $U(s; s) = I$, let  $ 0 \le \lambda_r \le \cdots \le \lambda_1$ denote the corresponding eigenvalues. A computation similar to that performed in the proof of Theorem \ref{thm:MultiRiccati}  below shows that these points perform the joint diffusion process:
\begin{align}
\label{eigsfornorm}
  d \lambda_i & = \frac{2}{\sqrt{\beta}} \lambda_i d b_i + ( - a + \frac{2}{\beta} ) \lambda_i  dt
                            + \lambda_i \sum_{j \neq i} \frac{\lambda_i + \lambda_j}{ \lambda_i  - \lambda_j}   dt.                    
\end{align}
For the stated almost sure conclusion, \eqref{eigsfornorm} is likewise viewed as a two-parameter family of processes $\lambda_i = \lambda_i(t; s)$ for $t \ge s$, all coupled through being run on the same Brownian motions $b_1, \dots, b_r$.

Passing to logarithms $\gamma_i = \log \lambda_i$, the upper and lower paths 
can bounded as in:
\begin{equation}
\label{lowestgamma}
   \gamma_1(t;s)  \ge  \frac{2}{\sqrt{\beta}} | b_1(t) - b_1(s) | + ( - a  + r-1) (t-s),
\end{equation}
and
\begin{equation}
\label{largestgamma}
  \gamma_r(t;s)  \le  \frac{2}{\sqrt{\beta}} | b_r(t) - b_r(s) | +  ( - a  - r+1)(t-s).
\end{equation}
These estimates follow upon substituting  $ \sum_{j \ne 1} \frac{\lambda_1+ \lambda_j }{\lambda_1 - \lambda_j}  \ge r-1$ or  
$ \sum_{j \ne r} \frac{\lambda_r+\lambda_j}{\lambda_r - \lambda_j}  \le -r+1$, respectively, into \eqref{eigsfornorm}.  Next, since
$$ 
\sum_{j \neq k} 
 \frac{\lambda_{k}+ \lambda_j}{\lambda_{k} - \lambda_j} 
 - \sum_{j \neq k+1 } \frac{\lambda_{k+1}+ \lambda_j}{\lambda_{k+1} - \lambda_j} \le 2 \, \frac{\lambda_{k} + \lambda_{k+1}}{\lambda_{k}-\lambda_{k+1}},
$$
we have that $\gamma_{k} - \gamma_{k+1}$ is path-wise bounded by the solution of
\begin{equation*}
  dz_t = db_t + 2 \, dt +  { 4 }{(e^{z_t}- 1)^{-1}} \, dt,  
\end{equation*}
again started from $t=s$.  Here $b_t = \frac{2}{\sqrt{\beta}} ( b_{k}(t)  - b_{k+1}(t))$.  Since $\frac{1}{e^z -1} \ge \frac{1}{z}$, the process $z_t$ is almost surely positive. It is also bounded below by $b_t +2t$, and so there is an almost surely finite time $\tau$ after which $z_t > t$. These observations imply that $\int_0^{\infty} \frac{dt}{e^{z_t}-1}$ is finite with probability one.     
Hence,
\begin{equation}
\label{incrementbound}
   \gamma_{k} - \gamma_{k+1}  \le    \frac{2}{\sqrt{\beta}} | b_k(t) - b_k(s) |  + \frac{2}{\sqrt{\beta}} | b_{k+1}(t) - b_{k+1}(s) |  + 2 (t-s) +  C,
\end{equation}
with an almost surely finite $C = C(b_k, b_{k+1})$.

Putting together  \eqref{lowestgamma} and \eqref{incrementbound}, along with the law of the iterated logarithm, we find that 
$$
 \gamma_r(t;s) \le \gamma_1(t;s) + \sum_{k=1}^{r-1} (\gamma_{k+1}(t;s) - \gamma_k(t;s)) \le (-a +r-1) (t-s) + C (1+ (t-s)^{\frac{1}{2}+})  
$$ 
with another almost surely finite constant $C = C(b_1, \dots, b_r)$. The conclusion is that  the integral
$\int_0^{\infty} \int_s^{\infty} $ $ e^{-rt} || \A_s^{-1} \A_t  ||^2  dt  ds $ converges along with
 $\int_0^{\infty} \int_s^{\infty} e^{-rt} e^{(r-1-a) (t-s)} dt ds$. 
\end{proof}

For later we extract the following from the proof of Lemma \ref{lem:ContinuumHSBound}.

\begin{corollary}  
\label{cor:Norms}
There is a finite random constant $C$ such that, for $s\le t$, 
\begin{equation}
\label{calnorms}
   || \A_s^{-1} \A_t || \le  C  e^{ \frac{ r -a -1}{2} \, (t-s)   + C (t-s)^{\frac{1}{2}+} },
 \end{equation}
holds almost surely. Back in the original coordinates this implies that, also almost surely and with 
$0 \le x \le y \le 1$, 
\begin{equation}
\label{firstnorms}
   || A_x^{-1} A_y ||  \le  C'  \left( \frac{1-y}{1-x} \right)^{-\frac{1}{2} + \frac{a+1}{2r} - \epsilon},
\end{equation}
for $\epsilon > 0$ any  chosen (small, deterministic) constant and (a finite, random) $C' = C'(C, \epsilon)$.
\end{corollary}


\subsection{Rewriting the matrix product}
\label{sec:Euler}

A bit of convenient preprocessing casts the matrix product $A_n(x)$, see again \eqref{eq:An} and \eqref{discreteAnEq0}, as an approximate Euler scheme for the claimed limiting sde \eqref{Aeq}.

\begin{lemma} 
\label{lem:Euler}
 With
 $\triangle
= \frac{r}{n}$ there is the representation
\begin{equation}
\label{discreteAnEq1}
   A_n(x + \triangle) - A_n(x) = A_n(x) 
    \left( \frac{1}{\sqrt{ n (1-x)}} G_{[ n x/r ]} +  \frac{(- a +\frac{1}{\beta})}{2 n(1-x)} {I}  + \E_{[ n x/r ]} \right).  
\end{equation}
Here $G_k = G_{n,r}$ is an independent sequence of approximate $\mathbb{F}$-Gaussian matrices. All off-diagonal elements are independent unit $\mathbb{F}$-Gaussians, while the diagonals are sums of independent $\chi$ variables: 
\begin{equation}\label{Gdiag}
   (G_k)_{ii} = (D_k)_{ii} + (O_k)_{ii} - \ev [ (D_k)_{ii} + (O_k)_{ii} ],
\end{equation}
recall the definitions \eqref{ondiag} and \eqref{offdiag}.
The error matrices $\E_k = \E_{n,k}$ form another  independent sequence. Each has a deterministic plus noise decomposition $\E_k  = \E_k^0 + \E_k^1$ satisfying:
\begin{equation}
\label{ErrorConditions}
 \| \E_k^0 \| \le \frac{c}{(n - rk)^{3/2}}, \quad \ev \E_{k}^1 =0,  \quad \ev  \| \E_{k}^1 \|^{2p}  \le \frac{c}{(n-rk)^{2p}}
\end{equation}
for integer $p \ge 1$
with constants $c = c(\beta,a, r, c',p)$ granted that $n -rk \ge c' \gg 1$.
\end{lemma} 

One first  byproduct of the above  (coupled with say \cite[Theorem 11.2.3]{SV}) is the advertised convergence of $A_n(x)$ to $A_x$. The representation \eqref{discreteAnEq1} will also be key for the norm control of the associated kernel.

\begin{corollary}
The Markov chain $A_n(x)$ converges pathwise (in the usual Skorohod topology of uniform on compacts in
$[0,1)$) to the diffusion $A_x$ defined by \eqref{Aeq}. 
\end{corollary}

\begin{proof}[Proof of Lemma \ref{lem:Euler}]
We have to show that 
\begin{equation}
\label{IncrementExpansion}
O_{k}^{} D_{k+1}^{-1} = I+   \frac{1}{\sqrt{n - rk}} G  +  \frac{a -\frac{1}{\beta}}{2 (n - rk)} {I} + \E,
\end{equation}
with random matrices $G$ and  $\E$ having the stated properties.  
For convenience we are shifting the index on the $OD^{-1}$ product to $k$ from $(k+1)$.

Denote by $F_{k,b} = F_{n,k,b}$ the $r \times r$ diagonal matrix with nonzero entries $ \frac{1}{\sqrt{\beta}} \ev \chi_{\beta (n + b - r k - i+1)}$, $i =1,2,\dots, r$. Then, recalling the definitions of the $D$ and $O$ blocks \eqref{ondiag}-\eqref{offdiag},  we write
\begin{equation}
\label{IncrementDecomposition}
  O_k^{} D_{k+1}^{-1} := (F_{k,0}  + \tilde{H} ) ( F_{k,a} - H )^{-1},
\end{equation}
making  two more definitions. $H$ is upper triangular with independent $F$-Gaussians off diagonal and centered negative $\chi$ variables (pairing with 
$F_{k,a}$ entry-wise) on diagonal, and $\tilde{H}$ is lower triangular, again with independent $F$-Gaussians off diagonal and  centered $\chi$ variables (pairing  now with 
$F_{k,0}$) on diagonal. Extracting
$$
   G := H + \tilde{H}
$$
is the guiding principle.

By the resolvent identity we have that, 
\begin{equation}
\label{resolvent}
  (F_a - H)^{-1} = F_a^{-1} + F_a^{-1}  H F_a^{-1} + F_a^{-1} (H F_a^{-1})^2 + (F_a - H)^{-1} (H F_a^{-1})^{3},
\end{equation}
where we drop the $k$ indices from here on. Substituting the polynomial approximation inherent in \eqref{resolvent} into \eqref{IncrementDecomposition} gives
\begin{align}
\label{IncrementExpand1}
   (F_0 + \tilde{H}) & ( F_{a} - H )^{-1}  ( I - (H F_a^{-1})^{3})  \\
    &  =   F_0 F_a^{-1} + \tilde{H} F_a^{-1}  + F_0^{} F_a^{-1}  H F_a^{-1} +
       F_0^{} F_a^{-1}  \ev (H F_a^{-1})^2   \nonumber \\
                                                            &  + \tilde{H} F_a^{-1} H F_a^{-1} +  \tilde{H} F_a^{-1} (H F_a^{-1})^2 + 
                                                             F_0^{} F_a^{-1} (  (H F_a^{-1})^2  - \ev (H F_a^{-1})^2 ). \nonumber 
\end{align}
The first line on the right hand side of \eqref{IncrementExpand1} contains the leading order in  \eqref{IncrementExpansion}, while the second line is mean zero ($H$ and $\tilde{H}$ are independent with mean zero entries) and will be absorbed into $\E$.

To deal with the first line we need the estimates
\begin{align}
\label{eq:Fs}
       \|  F_0^{} F_a^{-1} -  (1 - \frac{a}{2(n-k)} )  I  \|  &=  O \left(\frac{1}{(n-rk)^{3/2}} \right),  \\              
      \quad  \|  F_{a}^{-1} - \frac{1}{\sqrt{n-rk}} I \|  & =   O \left(\frac{1}{n-rk} \right),  \nonumber
\end{align}
in the first three terms, and, given that $ \ev (H F_a^{-1})^2  = F_a^{-2} \ev H^2$, 
\begin{align}
\label{eq:HFs}
      \quad  \|  \ev H^2 -  \frac{1}{2 \beta} I \|  & = O \left(\frac{1}{n-rk} \right),
\end{align}
in the last term. Of course, the same estimates would hold for $\tilde{H}$ replacing $H$ in the last display. 
All of this comes down to the simple appraisals
\begin{equation}
\label{chiMeanVar}
   \ev \chi_t = \sqrt{t}  - \frac{1}{4\sqrt{t}} + O \left(  \frac{1}{t^{3/2}} \right), 
\quad
     \mbox{Var} \, \chi_{t}  = \frac{1}{2} + O \left(\frac{1}{t} \right),
\end{equation}
for $t$ large enough. In summary,
\begin{align}
\label{summary1}
  F_0^{} F_a^{-1} + \tilde{H} F_a^{-1}  & + F_0^{} F_a^{-1}  H F_a^{-1} +   F_0^{} F_a^{-1}  \ev (H F_a^{-1})^2   \\
                            & = I  + \frac{1}{\sqrt{n-rk}} (H + \tilde{H})  + \frac{(- a +1/\beta)}{2(n-rk)} \mathcal{I} + \E', \nonumber
\end{align}
with diagonal $\mathcal{I} = I + O( (n-rk)^{-1/2})$ and a mean-zero $\E'$  satisfying the conditions laid out for the 
$\mathcal{E}_k^1$ terms.
For instance, $\E'$ contains terms like $\tilde{H} (F_a^{-1} - \frac{1}{\sqrt{n -rk} } I)$ with 
expected norm powers controlled by the second line of \eqref{eq:Fs} and the fact that the entries of $H$ and $\tilde{H}$ have bounded moments of all orders. For the latter the estimate \eqref{chiMeanVar} is supplemented with
$
    \ev | \chi_t - \ev \chi_t |^{2p} \le (c p)^p,
$
for $t > c'$ and $c = c(c')$.

Returning to the second line on the right hand side of \eqref{IncrementExpand1}, we have already remarked that all terms have mean zero and the expected norm  estimates follow similarly to those we just described. Now use $ \| F_0 F_a^{-1}  \| \le c$ and  $ \| F_a^{-1} \| \le \frac{c}{\sqrt{n-rk}}$ along with the moment control on the entries of $H$ and $\tilde{H}$.

Finally consider the remainder from \eqref{resolvent} and  \eqref{IncrementExpand1}, or 
$$
    \Xi :=   (F_0 + \tilde{H})  ( F_{a} - H )^{-1}   (H F_a^{-1})^{3}.  
$$
We will show that: as always in the regime of large $n-rk$, 
\begin{equation}
\label{lastI}
     \left[  F_0 \ev (F_a - H)^{-1} (H F_a^{-1})^3 \right]_{ij} = O \left( \frac{\delta_{ij}}{(n-rk)^{3/2}} \right), 
\end{equation}
that is, it is diagonal with diagonal entires bounded in kind, and that
 \begin{align}
 \label{lastE}
   \ev \|  \tilde{H}   ( F_{a} - H )^{-1}   (H F_a^{-1})^{3} \|^{2p} & =  O  \left( \frac{1}{(n-rk)^{4p}} \right), \\
   \ev \| F_0  \Bigl( (F_a - H)^{-1} (H F_a^{-1})^3  - E (F_a - H)^{-1} (H F_a^{-1})^3 \Bigr) \|^{2p}  & =  O  \left( \frac{1}{(n-rk)^{3p}} \right). \nonumber
 \end{align}
 With these established, $F_0 \ev (F_a - H)^{-1} (H F_a^{-1})^3$ can be added to $\E^0$ and the mean-zero matrices inside the norms in \eqref{lastE}
 can be absorbed into  to $\E^{1}$ and the proof will be complete.
 
 Starting with \eqref{lastI}, that $\ev [ (F_a - H)^{-1} (H F_a^{-1})^3 ]$ is diagonal follows from the Cayley-Hamilton theorem along with the fact that $E H^p$ is diagonal for any 
 integer power. To see that latter, expand out $(H^p)_{ij} = H_{i_0 i_1} H_{i_1 i_2} \cdots H_{i_{p-1} i_p}$, with $i_0 = i, i_p = j$ and note that unless $i=j$ there is at least one
 non-repeated (mean-zero) $H_{i_k i_{k+1}}$. This is simply because $H$ is upper triangular and so $i_k \le i_{k+1}$ along the expansion. The estimate itself can also be seen term-wise, by say writing out $(F_a - H)^{-1}$ using Cramer's rule. The bounds in \eqref{lastE} are similar. After an application of Cauchy-Schwarz both bounds come down to the inequality $\ev \| (F_a - H)^{-1} \|^{4p}$ $\le \frac{c}{(n-rk)^{2p}}$ which yet again follows from expanding things entry-wise and using  $\ev |H_{ij}|^q = O(1)$ for all $i,j$ and $q$.
 \end{proof}

\subsection{Tightness up to good times}
\label{sec:drectattempt}

We prove that for any $\epsilon > 0$, there is a $c \gg 1$ such that 
\begin{equation}
\label{eq:DiscreteNormBound}
      \|  A_n(x) \|^2 \le C_n (1-x)^{ -1 +  \frac{a+1}{r} - \epsilon}, \quad \mbox{ for } x \le 1 - c \frac{\log n}{n},
\end{equation}
where $C_n = C_n(\epsilon)$ is a tight sequence of random variables. Considering partial sums we will then conclude that

\begin{corollary}
\label{cor:DiscreteBound}
For $ \epsilon > 0$ it holds
\begin{equation}
  \label{eq:DiscreteNormBound2}
   \| A_n(x)^{-1} A_n(y) \| \le C_n  \left( \frac{1-y}{1-x} \right)^{-\frac{1}{2} +  \frac{a+1}{2r} - \epsilon}  \quad \mbox{ for } 0 \le x \le y \le 1 - c \frac{\log n}{n},
\end{equation}
with a  deterministic $c = c(\epsilon)$ and  a tight random sequence $C_n = C_n(\epsilon)$.
\end{corollary}

This should be compared to the estimate in Corollary \ref{cor:Norms} for the continuum case, and finishes  Step 2 of the proof of Theorem \ref{thm2crux}. Note that a tight random upper bound of the form $ \sqrt{n} D_{[nx/r]}^{-1}  \le C_n' (1-x)^{-1/2} I$ poses no difficulty, and also that the righthand side of \eqref{eq:DiscreteNormBound2} is in $L^2 [ \{0 < x < y < 1 \}, \frac{1}{1-x} dx \times dy ]$  as long as $\epsilon < \frac{a+1}{2r}$.

As for \eqref{eq:DiscreteNormBound}, the proof is inspired by \cite{Newman}.
From Lemma \ref{lem:Euler} we can express $A_n(x) = \prod_{k=1}^{[nx/r]}  Z_k$
with
\begin{equation}
\label{eq:Zthing}
     Z_k = I +  \frac{1}{\sqrt{ n- rk}} G_{k} +  \frac{(- a +\frac{1}{\beta})}{2 (n-rk)} I + \E_k,
\end{equation}
where the $G_k$ are  independent approximate $\mathbb{F}$-Gaussian matrices and  the error matrices $\E_k^{}$
possess the decomposition
$
  \E_k = \E_k^0 + \E^1_{k}
$  
in which $\E_k^0$ are (deterministic) diagonal matrices satisfying $\| \E_k^0 \| \le c (n-rk)^{-3/2}$ and $\ev \E_k^1 = 0, \ev \| \E_k^1 \|^{2p}  \le c(n-rk)^{-2p}$, for $n-rk > c'(c)$, $p \ge 1$.

Toward controlling the norm of $A_n(x)$, choose a fixed  $v \in S^{r-1}(\mathbb{F})$ and consider
\begin{equation}
\label{eq:InnerProductThing}
   v^{\dagger} A_n(x) A_n(x)^{\dagger} v = v^{\dagger} ( \prod_{k=1}^{[nx/r]} Z_k ) (\prod_{k=1}^{[nx/r]} Z_k)^{\dagger} v =  \prod_{k=1}^{[nx/r]} v_k^{\dagger} Z_k^{} Z_k^{\dagger} v_k^{},
\end{equation}
where the introduced vectors $v_k \in S^{r-1}(\mathbb{F})$ are defined by
$$
  v_k  =  
  \frac{ (Z_{k-1}^{} Z_{k-1}^{\dagger}) (Z_{k-1}^{} Z_{k-2}^{\dagger} )\cdots (Z_1^{} Z_1^{\dagger}) 
  v}{ \|  (Z_{k-1}^{} Z_{k-1}^{\dagger}) (Z_{k-1}^{} Z_{k-2}^{\dagger} )\cdots (Z_1^{} Z_1^{\dagger})  v \|^{1/2}}
  \mbox{ for } k  > 1, \quad v_1 =v. 
$$  
In particular, each $v_k$ is $\mathcal{F}_{k-1}$ measurable. With this set-up,  \eqref{eq:DiscreteNormBound} will follow from:

\begin{proposition} Set
\begin{equation}
\label{eq:Qdef}
Q_x  = Q_x(v) = \sum_{k=1}^{[nx/r]}   \log \Bigl(  v_k^{\dagger} Z_k^{}  Z_k^{\dagger} v_k^{} \Bigr) - \frac{(r-a-1 +  \frac{1}{2} {\epsilon})}{n-rk},  
\end{equation}
the centered logarithm of \eqref{eq:InnerProductThing},
and  $ h(x) = 
  (1- \log(1-x))^{-1/2+}$. Then, with choice of $c$ depending on $\epsilon > 0$, it holds that
\begin{equation}
\label{OneLineTight}
   0 \le  \sup_{x \le 1- c \frac{\log n}{n}}  h(x) Q_x  \le C_n
\end{equation}
for $C_n$ a tight sequence (dependent on $c$ but independent of $v$).
\end{proposition}

Indeed, just write: for $x$ in the prescribed range and $u_{\ell}$, $\ell =1 \dots, r$ the standard basis vectors,
\begin{align*}
  \tr (A_n(x)^{\dagger} A_n(x) ) & = e^{\sum_{k=1}^{[nx/r]} \Bigl( \frac{r-a-1+ \epsilon/2}{n-rk} \Bigr)}  \sum_{\ell=1}^r e^{Q_x( u_{\ell})} \\
                                                 &  \le 2 r e^{ (1 - \frac{a+1}{r} + \frac{\epsilon}{2r}) \log \frac{1}{(1-x)} } e^{ C_n (1 +   ( \log  \frac{1}{(1-x)} ) ^{1/2+}  )} \\
                                                 & \le  C_n' (1-x)^{-1 + \frac{a+1}{r} - {\epsilon}}. 
\end{align*}
The first inequality uses \eqref{OneLineTight} in the second factor and a rough upper bound in  the first (valid for large enough $n$). The second inequality just introduces a new tight random sequence $C_n'$, based on $C_n$, $\epsilon$, and $r$.

\begin{proof}  
First, since we only require a tight random constant bound, we are free to restrict to an event $\Omega_n$ with probability tending to one as $n \rightarrow \infty$. Let $\Omega_n$ be the event on which all the $\mathbb{F}$-Gaussians appearing in the collection of $Z_k$ matrices $k \le 1 - 
c \frac{\log n}{n}$ are bounded in norm by $c' \log n$, and all the $\chi$ random variables are  within $c' \log n$ of their mean. The claim is that by choice of $c'$ the probabilities $\pr (\Omega_n)$ are in fact summable, and on $\Omega_n$ it holds that
\begin{equation}
\label{Znorm}
  | v^{\dagger} Z_k^{} Z_{k}^{\dagger} v -1| \le   \frac{c'' \sqrt{\log n}}{\sqrt{n-k}},
\end{equation}
where $c''$  depends on $c'$ only. In the next subsection (``Throwing out the tail") we prove a version of this claim in a more difficult setting $-$ see in particular \eqref{AAbound1} and its subsequent proof $-$  and so do not repeat the argument here.

By taking  $c$ sufficiently large the inequality \eqref{Znorm} can be taken to read 
$  | v^{\dagger} Z_k^{} Z_{k}^{\dagger} v - 1 | \le  \delta$ with a choice of $\epsilon > 0$ for  all  $k \le n- c \log n$. This leads to the bound: on $\Omega_n$ and for $x \le 1 - c \frac{\log n}{n}$,
\begin{equation}
\label{PrelimQ}
  \sum_{k=1}^{[nx/r]}   \log \Bigl(  v_k^{\dagger} Z_k^{}  Z_k^{\dagger} v_k^{} \Bigr)  \le  \sum_{k=1}^{[nx/r]} ( v_k^{\dagger} Z_k^{} Z_{k}^{\dagger} v_k - 1) - \frac{(1-\delta)}{2}   \sum_{k=1}^{[nx/r]} ( v_k^{\dagger} Z_k^{} Z_{k}^{\dagger} v_k - 1)^2,
\end{equation}
as $\log (1+x) \le x - \frac{1}{2}(1-\delta) x^2$ for $x < \delta$.

Next denote
\begin{equation}
\label{eq:Mthing}  
Z_k^{} Z_k^{\dagger} = I +   X_k  +  Y_k  +\frac{(- a +\frac{1}{\beta})}{ (n-rk)} I + \E_k^{'}
\end{equation}
for
\begin{align}
\label{XY}
  X_k =  \frac{1}{\sqrt{n-rk}} (G_k^{} + G_k^{\dagger}),   \quad Y_k = \frac{1}{n-rk}  G_k^{} G_k^{\dagger}, 
\end{align}
and
\begin{align}
\label{ErrorNew}
  \E_k^{'} =  ( \E_k + \E_k^{ \dagger} ) + \E_k^{} \E_k^{ \dagger}    +  \frac{(- a +\frac{1}{\beta})^2}{4 (n-rk)^2} I  +\mathrm{ etcetera}, 
\end{align}
yet another error term. 
Also introduce the
shorthand:
$ X_k(v) =   v^{\dagger} X_k v,$  
$Y_k(v) = v^{\dagger} Y_k v,$
 and so on 
for any $v \in S^{r-1}$.

By choosing $\delta = \frac{\beta}{4} \epsilon$ the inequality \eqref{PrelimQ} implies that
\begin{align}
\label{Qbound}
Q_x  \le    \sum_{k=1}^{[nx/r]}  &  \left( X_k(v_k) +  ( Y_k(v_k)  -  \frac{(r-1 - \frac{1}{\beta})}{ (n-rk)}  ) +
      (1-\delta) (\frac{2}{\beta (n-rk)} - \frac{1}{2} X_k(v_k)^2) \right) \\
      & +  \sum_{k=1}^{[nx/r]}  \Bigl( \E_k'(v_k) + W_k(v_k)^2 + 2 (1-\delta)   X_k(v_k) W_k(v_k) \Bigr), \nonumber
\end{align}
in which $W_k(v) = Y_k(v) + \frac{(- a +\frac{1}{\beta})}{ (n-rk)}  + \E_k^{'}(v)$. Intuitively, line one of \eqref{Qbound} produces the leading order with the terms in line two having greater built-in decay. 

As for that first line, we consider each term separately. That is, we demonstrate tight random constant bounds on:
\begin{equation}
\label{Xterm}
   \sup_{x\le 1 - c \frac{\log n}{n}}  h(x) \left( \sum_{k=1}^{[nx/r]} X_k(v_k)  \right), 
\end{equation}  
and
\begin{align}
\label{XXYterms}
 \sup_{x\le 1 - c \frac{\log n}{n}} &  h(x) \left( \sum_{k=1}^{[nx/r]}     Y_k(v_k)  -  \frac{(r-1 - \frac{1}{\beta})}{ (n-rk)}   \right), \\
  \sup_{x\le 1 - c \frac{\log n}{n}}  &   h(x) \left( \sum_{k=1}^{[nx/r]}  \frac{4}{\beta (n-rk)} - X_k(v_k)^2 \right).
  \nonumber
 \end{align}
Of these, the near Gaussian sum in \eqref{Xterm} is the most delicate,  requiring the type of quantitative iterated log estimate carried out in the end of Section 2 of \cite{RR}. The $Y_k$ and $X_k^2$ sums can be dealt with more crudely as follows.

Using \eqref{chiMeanVar} one easily finds that
$$
   \ev Y_k(v) = \frac{r-1-\frac{1}{\beta}}{(n-rk)} + O\left( \frac{1}{(n-rk)^2} \right), \quad
   \ev X_k(v)^2 =  \frac{4}{\beta (n-rk)} + O\left( \frac{1}{(n-rk)^2} \right)
$$
uniformly in $v \in S^{r-1}$. And since $\sum_{k=1}^{[(n- c \log n)/r]} \frac{1}{(n-rk)^2} = O \left( \frac{1}{\log n} \right)$, the sums in \eqref{XXYterms} can be replaced by $\sum_{k=1}^{[nx/r]} \eta_k(v_k)$ with
$$
     \eta_k(v_k) = Y_k(v_k) -  \ev Y_k(v_k)  \quad \mbox{ or }   \quad   \eta_k(v_k) = \ev X_k(v_k)^2 -  X_k(v_k)^2, 
$$
respectively. In either case it holds that  $\ev \eta_k(v_k) = 0$, $\ev  \eta_\ell(v_\ell) \eta_k (v_k) = 0$ for $\ell \neq k$, and $ \ev  | \eta_k(v_k) |^2  \le \frac{c'}{(n-rk)^2}$  with a $c'$ independent of $v_k$ (and so the original $v$). Taking second moments and bounding the sup by another sum, we have that
\begin{align*}
  \ev \sup_{x\le 1}  \left( h(x) \sum_1^{[ nx/r ] } \eta_k(v_k) \right)^2  & \le \sum_{k=1}^{[ n/r ]} h^2(k/nr)
   \ev \left(\sum_{\ell=1}^k    \eta_k(v_k)\right)^2 \\
      & \le c' \sum_{k=1}^{ [ n/r ]} h^2(k/nr) \sum_{\ell=1}^k  (n-r \ell)^{-2}  \\
      & \le c'' \int_0^1 \frac{dy}{(1-y) (1+ |\log (1-y)|^{1+})},
\end{align*}
which is finite.

Line two of \eqref{Qbound} requires more of the same. In fact, the sum over $v_k^{\dagger} (\E_k^{} + \E_k^{ \dagger} ) v_k^{}$ (the first term in $\E_k^{'}(v_k)$) can be treated by the exact same $L^2$ bound just performed. Subsequent terms simply require taking higher moments, making use of the estimates \eqref{ErrorConditions} from Lemma \ref{lem:Euler}.

Returning to 
$\sum _{k=1}^{[ nx/r ]} X_k(v_k)$, the first ingredient is the following bound for chi-variables (appearing in the diagonals of the $X_k$ matrices):
\begin{equation}
\label{chi_mgf}
\ev \left[ e^{\lambda  \chi_t} \right] \le e^{\lambda \ev \chi_t + \frac{\lambda^2}{2}  } \mbox{ for all } \lambda \in \RR.
\end{equation}
This follows from the log-concavity of the $\chi_t$ density, which yields a Log-Sobolev inequality.  With this in hand, the well known Herbst's argument  gives \eqref{chi_mgf}, see for example  \cite{Ledoux}. 

Then,  for $X = $ a given $X_k$  (and likewise  $G$ denoting the corresponding approximate Gaussian matrix $G_k$) and fixed $v \in S^{r-1}$,  
\begin{align}
\label{term_mgf}
  \ev \left[ e^{\lambda X(v)} \right]  & = \ev \left[ e^{ \sum_i  \frac{2 \lambda}{\sqrt{n-rk}}  G_{ii} |v_i|^2 }   \right] 
    \ev \left[
  e^{ \sum_{i\neq j}   \frac{\lambda}{\sqrt{n-rk}}  (G_{ij} v_i^{\dagger} v_j^{} +  G_{ij}^{\dagger} v_i^{} v_j^{\dagger})    } \right]  \\
  & \le e^{  \sum_i  \frac{4 \lambda^2}{ \beta ({n-rk})}  |v_i|^4 }  e^{ \sum_{i \neq j} \frac{2 \lambda^2}{ \beta ({n-rk})}  |v_i|^2 |v_j|^2 } \nonumber \\
  & \le e^{ (4/\beta) \lambda^2 \frac{1}{(n-rk)}}. \nonumber
\end{align}
Recall the definition of $G$ in \eqref{Gdiag} from Lemma \ref{lem:Euler}: the first overestimate above comes from using
\eqref{chi_mgf}  in the first factor of \eqref{term_mgf}. The outcomes is that
\begin{equation}
\label{EXPmartingale}
   \Xi_m := e^{ \lambda  \sum_{k=1}^m  X_k(v_k) - (4/\beta) \lambda^2 \sum_{k=1}^m \frac{1}{n-rk}}, \mbox{ for } m = 1, 2, \dots, [ n/r ]
\end{equation}
is a supermartingale. Now Corollary 4.2 of \cite{DKL} implies that: with probability one, 
\begin{equation}
\label{LASTIL}
  \limsup_{n \rightarrow \infty, x_n \rightarrow 1} \frac{ \sum_{k=1}^{ [nx_n/r]}  X_k(v_k)}{ \sqrt{ ( \sum_{k=1}^{[nx_n/r]} \frac{1}{n-rk} ) \log \log (\sum_{k=1}^{[nx_n/r]}  \frac{1}{n-rk})   }} \le c(\beta)   
 \end{equation}
with a constant $c(\beta)$ that can be explicitly worked out.  Here $x_n$ indicates any sequence tending to one as $n$ tends to infinity.
This is more than is needed to complete the proof. Here we are applying Condition 4.1 of the reference to \eqref{EXPmartingale} with $\Phi_r(s) = \frac{1}{2} s^2$ and
$B_t = B_{n,t_n} = $ $2 \sqrt{  (2/\beta)\sum_{k=1}^{[nt_n/r]} \frac{1}{n-rk}}$ with $t_n \rightarrow 1$. The obvious adaptation of the ideas there to the ``triangular array" of variables here yields \eqref{LASTIL}.
\end{proof}

\subsection{Throwing out the tail}
\label{sec:ThrowOutTail}

We establish Step 3 of the proof of Theorem \ref{thm2crux}.

\begin{lemma} 
\label{lem:ThrowOutTail}
For any large enough constant $c$,
\begin{equation}
\label{taileq}
   \iint\limits_{ {0 \le x \le y \le 1},  \, { y \ge 1- c \frac{\log n}{n} } } \| k_{n,r}(x,y) \|^2 \, dy dx  \rightarrow 0,
\end{equation}
in probability. 
\end{lemma}

\begin{proof} As noted already, the scaled pre-factors $ \sqrt{ n(1-x) }{D}_{\lceil (n/r) x \rceil}^{-1}$ are bounded by a tight random constant multiple of the identity.  
We can then consider the slightly simplified kernel 
$$  \tilde{k}_{n,r} (x,y) =  \frac{1}{\sqrt{1-x}} \prod_{k= [ nx/r ]}^{ [n y/r] } O_k D_{k+1}^{-1} =  
\frac{1}{\sqrt{1-x}} A_n(x)^{-1} A_n(y)
$$ 
in place of $k_{n,r}(x,y)$.  

The proof comes down to constructing events ${\Omega}_n$ with probability tending to one on which 
\begin{equation}
\label{AAbound1}
 \|  A_n(x)^{-1} A_n(y) \| \le  \prod_{n x \le  k \le ny}  ( 1 +    c' \sqrt{ \frac{\log \log n}{{n-k}}} ),  
\end{equation}
for the prescribed  range of $x,y$ and  $c'$ is a numerical constant, independent of $k$, $n$ or $c$. 
This is used in conjunction with Corollary \ref{cor:DiscreteBound} of the previous section which yields 
events $\Omega_n'$ (there referred to as $\Omega_n$)  of probability close to one on which
\begin{equation}
\label{AAbound2}
  \| A_n(x)^{-1} A_n(y) \| \le c'  \left( \frac{1-y}{1-x} \right)^{- \frac{1}{2} + \frac{a+1}{2 r} -\epsilon},  \ \mbox{ for } x, y \le  1 -  c  \frac{\log n}{n},
\end{equation}
$\epsilon < \frac{a+1}{2r}$ and now another deterministic constant $c'$. We are translating the tight random constant bound in that statement to a deterministic one on 
a sequence of nice events. The goal then is to show that on the disjoint union of $\Omega_n^{}$ and $\Omega_n'$ the appraisal
\eqref{taileq} holds (for again $\tilde{k}_{n,r}$). That $c$ be large enough is required exactly because we use \eqref{AAbound2} with $\pr ({ \Omega_n'}^c) \ll 1$.

So, with both \eqref{AAbound1} and 
\eqref{AAbound2} in place, write:
\begin{align*}
   \iint\limits_{ {0 \le x \le y \le 1},  \, { y \ge 1- \delta _n} }  \,  \| \tilde{k}_{n,r}(x,y) \|^2  dy dx = \int_{1-\delta_n}^1 \int_x^1  \| \tilde{k}_{n,r}(x,y) \|^2 dy dx
     +  \int_0^{1-\delta_n} \int_{1-\delta_n}^1  \| \tilde{k}_{n,r}(x,y) \|^2  dy dx.
\end{align*}
with $ \delta_n = c \frac{\log n}{n}$.

For the first term we have the estimate on the right hand side of \eqref{AAbound1},
$$
      ||  A_n(x)^{-1} A_n(y) ||^2 \le   e^{  \sqrt{ \kappa_n } ( \sqrt{1-x} - \sqrt{1-y})}, \quad
\mbox{ with } \kappa_n = c'' n \log \log n, 
$$
for large enough $n$ and constant $c''$.
This gives
\begin{align}
\label{AAstep1}
   \int_{1-\delta_n}^1 \int_x^1 \frac{1}{1-x}    \| A_n(x)^{-1} A_n(y) \|^2  dy dx & 
   \le    \int_0^{\delta_n} \frac{1}{x} \int_0^x e^{ \sqrt{\kappa_n  x} - \sqrt{ \kappa_n y} }  dy dx   \\
  &  = 2
   \int_0^{\delta_n}  \left (\frac{e^{ \sqrt{ \kappa_n x}} -1}{\kappa_n  x} - \frac{1}{\sqrt{ \kappa_n   x}} \right) dx 
   \le  \,  \frac{e^{\sqrt{\kappa_n\delta_n }}  }{\kappa_n}, \nonumber
 \end{align}
 which tends to zero like $1/n^{1-}$.  The final inequality rests  on the fact that $\frac{e^{\sqrt{x}}}{x}$ is increasing for large $x$, while $e^{\sqrt{x}} \le 1 + \sqrt{x} +    O(x)$ for moderate values of $x$.
 
 As to the second term we have: 
  \begin{align}
  \label{AAstep2}
 \int_0^{1-\delta_n} \int_{1-\delta_n}^1  \frac{1}{1-x} & \| A_n^{-1}(x) A_n(y) \|^2  dy dx  \\
                        & \le \int_0^{1-\delta_n}  \frac{1}{1-x} \| A_n^{-1}(x) A_{n} (1-\delta_n)  \|^2 \, dx  
                        \int_{1-\delta_n}^1  \| A_n^{-1}(1-\delta_n)  A_n(y) \|^2 \,dy \nonumber \\
                        & \le  2 c' \delta_n^{- {1} + \frac{a+1}{ r} - 2\epsilon}   \int_0^{1-\delta_n}  \frac{dx}{({1-x})^{ \frac{a+1}{ r} - 2\epsilon}} 
                                \times \frac{e^{ \sqrt{\kappa_n \delta_n}}}{\kappa_n}.  \nonumber
 \end{align}                        
Here we used \eqref{AAbound2} in the first factor, and  the same estimate just employed in \eqref{AAstep1} in the second factor.
Completing the remaining integral in gives an upper bound of order
$ \delta_n^{\frac{a+1}{r} - 2 \epsilon}  \,  \frac{e^{  \sqrt{   \kappa_n \delta_n }}  }{\kappa_n \delta_n}$
which  decays like a small negative power of $n$ as long as the $L^2$ condition $\epsilon < \frac{a+1}{2r}$ holds.

It remains  go back and construct the events $\Omega_n$. 
Along with the standard Gaussian tail inequality in the form $\pr \left(  |g| \ge \gamma \right)  \le  \frac{e^{-\gamma^2/2}}{\gamma}$ we will also need that, for any $\chi_t$ random variable with $t \ge 1$,
 \begin{equation}
\label{chiBound}
    \pr \left(   | \chi_t  - \EE \chi_t |  \ge \delta \ev \chi_t \right)  \le 2 e^{- \delta^2 t/2}.   
\end{equation} 
This follows readily from the inequality \eqref{chi_mgf} along with the bounds
\begin{equation}
\label{chiBound0}
    \sqrt{t -\frac{1}{2}} \le \ev \chi_t \le  \sqrt{t}. 
\end{equation} 
The lower estimate requiring $t \ge 1$, the upper estimate is Jensen's inequality. Clearly \eqref{chiBound} is most useful for $t \uparrow \infty$.  For $t = O(1)$ we get by with 
\begin{equation}
\label{chiBound3}
 \pr \left( \chi_t \le \delta \right) \le  \delta^t,
\end{equation}
for small $\delta > 0$ and $t \ge 1$ (just integrate the density).

Now, dropping indices for a moment, we write $O_k = d + G$ and $D_k = \tilde{d} + G$ where $d$ and $\tilde{d}$ are diagonal matrices of the independent $\chi$ variables:
$$
{d}_{i} =  \frac{1}{\sqrt{\beta}} \chi_{\beta (n - rk  -i+1)} ,  \quad \tilde{d}_i = \frac{1}{\sqrt{\beta}} \chi_{\beta (n+a - r(k-1)-i+1)},
$$
 leaving  
$G, \tilde{G}$ to denote the strictly upper/lower matrices of independent Gaussians ($g_i \in G$, $\tilde{g}_i \in \tilde{G}$). Denote by 
$\Gamma_k$ the index set of all $d_i, \tilde{d}_i, g_, \tilde{g}_i$ variables occuring in $O_k$ and $ D_{k+1}$.
Decomposing as in,
\begin{align}
\label{resolventterm}
   O_k D_{k+1}^{-1} - I 
        & = (d \tilde{d}^{-1}  - I) + {G}  \tilde{d}^{-1} -   ( I + {G} {d}^{-1} )  d \tilde{d}^{-1}
         (I + \tilde{G} \tilde{d}^{-1} )^{-1}    \tilde{G} \tilde{d}^{-1},
\end{align}
we see that on the event
$$
  \Omega(\delta, k) = \left\{ |g_i|, |\tilde{g}_i| \le 2 \sqrt{\log \log n};  \,  d_i \le \sqrt{n-k} (1+\delta);   \, d_i,  \tilde{d}_i  \ge \sqrt{n-k}(1-\delta) 
                                      \right\}_{i \in \Gamma_k},                                     
$$
with $\delta = \delta_k = 4 \frac{ \sqrt{{\log \log n}}  }{\sqrt{n-k}}$ assumed less than $\frac{1}{2}$, there are the bounds
$$ 
   \| d \tilde{d}^{-1}  - I \| \le   \frac{8 \sqrt{\log \log n} }{\sqrt{n-k}}, \mbox{ and }
   \| {G} {d}^{-1} \| ,  \| {G}  \tilde{d}^{-1} \|,  \|  \tilde{G} \tilde{d}^{-1} \| \le  \frac{4 \sqrt{\log \log n} }{\sqrt{n-k}}.
$$
That is to say, 
\begin{equation}
\label{CheapNorm1}
   \|   O_k D_{k+1}^{-1}  - I  \| \le   c' \frac{ \sqrt{{\log \log n}}  }{\sqrt{n-k}},   \quad  \mbox{on } \cap_k \Omega(\delta_k, k) \mbox{ over }
   n-k \in [ 16 \log \log n, c \log n],
\end{equation}
with the lower bound on $n-k$ providing a uniform estimate  on the appearance of $ \| (I -  \tilde{G} \tilde{d}^{-1})^{-1} \| 
\le (1- \| \tilde{G} \tilde{d}^{-1} \|)^{-1}$ 
in the last term of \eqref{resolventterm}. Then by the Gaussian tail bounds, recall in particular \eqref{chiBound} coupled with \eqref{chiBound0},
\begin{align}
\label{AEventbound}
   \pr \Bigl( \cup \, \Omega(\delta_k, k)^c, \,   k \in [ {n} - c {\log{n}}, & \,  {n} - 16 \log \log n ] \Bigr)   
    \\ &  \le c' \log n  \left(   \frac{e^{-2 \log \log n }}{\sqrt{\log \log n}} + e^{- 8 \log \log n} \right), \nonumber
\end{align}
with a constant $c' = c'(r, c)$ on the right hand side.

Next, for $n-k$ of the same order as $\log \log n$, $O_k D_k^{-1}$ no longer concentrates well about the identity. Instead, we use that on $n-k \le 
16 \log \log n$, the typical Gaussian is now $O(\sqrt{\log \log \log n} )$ and on the event
$$
    \Omega'( k) =  \left\{ |g_i|, |\tilde{g}_i| \le 2 \sqrt{\log \log \log n};  d_i \le 2 \sqrt{n-k} , \tilde{d}_i \ge \frac{1}{4} \sqrt{n-k}  \right\}_{i \in \Gamma_k} 
$$
we can overestimate
\begin{align}
\label{roughboundCD}
   \|  O_k D_{k+1}^{-1} \|  & \le ( \| d \tilde{d}^{-1} \| +  \|  G \tilde{d}^{-1} \|  )  \, \| (I + \tilde{G} \tilde{d}^{-1} )^{-1} \| \\
                                        & \le c' \frac{\sqrt{\log \log \log n}}{\sqrt{n-k}} \times  \left( \frac{\log \log \log n}{n-k} \right)^{r/2}   
                                        \le c' \frac{\sqrt{\log \log n}}{\sqrt{n-k}}. \nonumber
\end{align}
The second inequality is just an $\infty$-norm bound on $(I + \tilde{G} \tilde{d}^{-1} )^{-1} $, the entries of which are polynomial (of degree at most $r$) in the entries of $  \tilde{G} \tilde{d}^{-1} $. The last inequality holds for $n$ big enough by adjusting the constant $c'$. And similar to 
\eqref{AEventbound},
\begin{align}
\label{AEventbound2}
   \pr \Bigl( \cup \, \Omega'( k)^c, \,   k \in  & [ {n} - 16 \, {\log \log{n}},   \,  {n} - \gamma ] \Bigr)   \nonumber \\
     &  \le c'' \frac{ \log \log  n }{\sqrt{\log \log \log  n}}   e^{-2 \log \log  \log n}   
      +  c'' \sum_{k =  \gamma}^{{16 \log \log n}}  e^{-\frac{1}{4} k },
\end{align}
which goes to zero for any $\gamma = \gamma_n \rightarrow \infty$. Note here we are we are using the bound \eqref{chiBound} with $\delta=1/2$  
(that is, 
$\Omega'(k) = \Omega'(\delta =1/2, k)$), along with the fact
$\EE \chi_t > \frac{1}{2} t$ for sufficiently large $t$, see again \eqref{chiBound0}. This gets us down to $n-k =O(1)$.

For the final stretch, we note that we have the same type of bound on $O_k D_{k+1}^{-1}$ as in  \eqref{roughboundCD} on an event $\Omega''(k)$ defined similar to $\Omega'(k)$ except replacing the conditions on $d_i, \tilde{d}_i$ with
$ \frac{d_i}{\sqrt{n-k}} \le  \sqrt{2 \log \log \log n}$ and $\frac{\tilde{d}_i}{\sqrt{n-k}} \le  1/\sqrt{ \log \log \log n}$. With $k$ now ranging from $n-\gamma$ to $n$, the adjustments  needed in what corresponds to the second term of \eqref{AEventbound2} are as follows.  The probability of the upper deviations of the maximum $d_i$ is bounded in the same manner as before by $ O ( \gamma e^{-\log \log \log n}) = O (\frac{\gamma}{\log \log n})$. For the probability of the lower deviations of the $\tilde{d}_i$ we use \eqref{chiBound3} to control this by $O( (\log \log \log n)^{-1/2} \sum_{k=1}^{\gamma} k^{k/2})$.  Since $\gamma \rightarrow \infty$ as slowly as we want, both of these bounds tend to zero as $n \rightarrow \infty$.

The advertised event $\Omega_n$ can then be defined as the intersection(s) of $\Omega(\delta_k, k)$, 
$\Omega'( k)$ and $\Omega''(k)$ over their respective ranges.
\end{proof}

\section{Riccati correspondence}
\label{section:Riccati}

We prove Theorems \ref{r1Riccati} and \ref{thm:MultiRiccati}, providing hitting time and PDE descriptions of the spiked hard edge laws, as well as Corollary \ref{cor:Large_a} which highlights the dependence of these descriptions on the ``non-degeneracy" parameter. 

\begin{proof}[Proof (sketch) of Theorem \ref{r1Riccati}]
The argument to identify the $q$ diffusion, with the appropriate starting point, is can be taken verbatim from Theorem 2 of \cite{RR} (with considerations for the boundary condition ``at infinity" fixed by the subsequent erratum).

To see the PDE connection is to note the (hyper-elliptic) differential operator 
$$
\mathcal{L}  := 
    \frac{\partial }{\partial \mu}   +  \frac{2}{\beta} c^2     \frac{\partial^2  }{\partial c^2}   + ( (a +\frac{2}{\beta}) c - c^2   - e^{-\mu}) 
      \frac{\partial }{\partial c},
$$ introduced in \eqref{r1PDE} is simply the space-time generator for the joint process $(x, q_x)$. Uniqueness follows from the standard martingale argument, starting from the observation that a solution to the backward equation $ \mathcal{L} G = 0$ produces 
a local martingale, $x \mapsto G(x, q_x)$. So then, with $\tau$ the passage time to zero, $G(\mu, c) = \EE_{(\mu,c)} G( \tau \wedge t, q_{\tau \wedge t})$. If $G$ is also bounded, the  $t \rightarrow \infty$ limit can be taken under the expectation. The resulting equality, along with the boundary conditions \eqref{BCs}, force $G$ to be the probability that the path never hits zero, that is, $G=F$. 

For the higher order distributions things are worked out inductively. Consider $F_1$, and let $\kappa_{\mu, c}(d\eta)$ denote the (improper) distribution of the
passage time to $-\infty$ starting from time/place $(\mu, c)$, allowing for the expression
$$
  F_1(\mu, c) = F(\mu, c) + \int_{-\infty}^{\infty} \kappa_{\mu,c}(d \eta) F(\eta, +\infty).
$$
From the same basic theory just employed, $\kappa$, which vanishes to the left of $\mu$, also satisfies the backward equation $\mathcal{L} \kappa =0$. The claimed boundary condition follows by observing that, as $c \downarrow -\infty$, $\kappa_{\mu,c}(d\eta)$ tends to the unit mass at $\eta =\mu$ and $F(\eta, +\infty)$ is continuous in $\eta$.
\end{proof}

\begin{proof}[Proof of Theorem \ref{thm:MultiRiccati}]
The approach mimics that for the one-dimensional case.
Introduce the truncated operator $\mathfrak{G}_L$ on $L^2[ \mathcal{M}, [0, L]]$ by
$$
  (\mathfrak{G}_L) f(x) = \int_0^L \int_0^{x \wedge y} \S(dz)   \M(dy)  f(y)+ C^{-1} \int_0^L  \M(dy)  f(y),
$$ 
for  $f: [0,L] \mapsto \mathbb{F}^r$. Since $\S_x$ and $\M_x$ are continuous and norm bounded  over $[0,L]$ any  
$f \in L^2 [\mathcal{M}, [0,L]] $ which also satisfies the eigenvalue problem $f(x) = \lambda \mathfrak{G}_L f(x)$
is readily seen to be $C^{1}$ and must also satisfy:
\begin{equation}
\label{vectorBCs}
     f'(0) = C f(0),  \quad  f'(x) =  \lambda  \mathcal{S}_x \int_x^{L} \mathcal{M}(dy)  f(y).
\end{equation}
This last condition holds entry-wise, so in particular $f'(L)$ is the zero vector. The important point is that the same estimates on the Hilbert-Schmidt norm of $k_r$ (Lemma
 \ref{lem:ContinuumHSBound})  shows that $\mathfrak{G}_L \rightarrow \mathfrak{G}$ almost surely in operator norm.  In particular, the desired eigenvalue counting function is the $L \rightarrow \infty$ limit of that for $\mathfrak{G}_L$.

The next step is to cast the condition that any fixed $\lambda$ is an eigenvalue (of $\mathfrak{G}_L$) in terms of an initial value problem, 
Since the paths of $\MM_x$ are (almost surely) bounded over $[0,L]$, any   $f \in L^2[\M]$ is also in $L^1[\M]$. And if such an $f$ satisfies the eigenvalue problem, it will then be continuous and, by repeating the argument, entry-wise differentiable. Taking derivatives we find
that
\begin{equation}  
\label{div1}
f'(x) =  \lambda  \mathcal{S}_x \int_x^{L} \mathcal{M}(dy)  f(y). 
\end{equation}
Following this by taking an It\^{o} differential,  we have the stochastic differential equation,
\begin{equation}
\label{multiRic1}
  df'(x)   =  \sqrt{2}  \A_x^{-\dagger} d \mathcal{B}_x  \A_x^{\dagger} \,  f'(x)  +   \left( (a + r-1 + \frac{2}{\beta} ) f'(x)   - \lambda e^{-rx} f(x) \right) dx,
\end{equation}
which can be closed by including $df(x) = f'(x)dx$ along with the equation for $A_x$. Note that $A_x$ is almost surely invertible for all $x \ge 0$.  Here $   \mathcal{B}_x = \frac{1}{\sqrt{2}}  ( B_x + B_x^{\dagger} )$ is the canonical $\beta=1,2,4$ invariant (or ``Dyson") Brownian motion, and we used the equation for  $\S_x$:
$$
\label{Seq}
  d \S_x =  \sqrt{2} 
  \A_x^{-\dagger} d \mathcal{B}_x \A_x^{-1}  +  (a +  r-1 + \frac{2}{\beta} ) \S_x dx.
  $$
  The latter follows from  It\^o's rule in the form
  $
  d M_x^{-1} = -  M_x^{-1} dM_x M_x^{-1}  +  M_x^{-1} dM_x M_x^{-1} dM_x M_x^{-1}.
$

The ``twisting" of the Brownian increment by $A_x^{\dagger}$ in \eqref{multiRic1} is a complication particular to $r>1$. Even without this, the fact that the noise appears multiplicatively seems to preclude reducing things directly to the classical (deterministic) matrix oscillation theory as was done for the multi-spiked soft edge in \cite{BV2}.  
On the other hand, by the Markov property and (path-wise) uniqueness for the system  $x \mapsto (A_x, f(x), f'(x))$, it cannot happen that
both $f$ and $f'$ vanish simultaneously. Further, by \eqref{div1} we see that $f'(L)= 0$ is precisely the condition for the predetermined $\lambda$ to be an eigenvalue of $\mathfrak{G}_L$. 

To make computations it is convenient to lift things to a matrix system. Define $F(x) \in \mathbb{F}^{r \times r}$ as the (unique) solution to
\begin{align}
\label{multiRic2}
  dF'(x)   &=  \sqrt{2}  \A_x^{-\dagger} d \mathcal{B}_x  \A_x^{\dagger} \,  F'(x)  +   \left( (a + r-1 + \frac{2}{\beta} ) F'(x)   - \lambda e^{-rx} F(x) \right) dx, \\
   d F(x) & = F'(x) dx, \nonumber
\end{align}
with initial conditions
 $$
 F(0) = I, \quad  F'(0) = C.
$$ 
  Evolving this system forward (including the evolution of $A_x$) for fixed $\lambda$, the outcome is that
$$ 
   \{ \lambda  \mbox{ is an eigenvalue of } \mathfrak{G}_L \} = \{ F'(L) \mbox{ is singular} \}. 
$$
Note that the uniqueness argument used for $(f, f')$ shows that $F$ and $F'$ cannot be singular at any common time.

The standard Riccati procedure would be to now introduce  $P(x) = F'(x) F^{-1}(x)$, which, away from the singular points of $F(x)$, solves
\begin{equation}
\label{matrixriccati}
   dP_x = \sqrt{2}  \A_x^{-\dagger} d \mathcal{B}_x  \A_x^{\dagger} P_x  +  (r-1) P_x  dx + \psi(P_x)  dx, \quad P_0 = C.
\end{equation}
Here the definition of
$$
\psi(P) = (a + \frac{2}{\beta} )P - P^2 - \lambda e^{-rx} I_r,
$$
from the statement of the theorem is simply lifted to act on matrices. While \eqref{matrixriccati} recovers the $r=1$ hard edge process (as it must), for $r > 1$  one must (again) include the auxiliary process $\A_x$ to close the equation.  Additionally, the hope would be to reduce the dimension of the system by translating the singular point count  of $F'(x)$ over $[0,L]$ to an eigenvalue count  of $P_x$.  But $P_x$ is not even say  symmetric in law.

The program is salvaged by introducing 
$
Q_x =  \A_x^{\dagger} P_x \A_x^{-\dagger},
$
which satisfies
\begin{equation}
\label{matrixriccati2}
  d Q_x =    dB_x Q_x  +  Q_x dB_x^{\dagger}  +  \left(   \mathrm{tr}(Q_x) I_r - \frac{1}{\beta} Q_x + (\frac{1}{\beta} -1) 
  \mathrm{diag}(Q_x)   + \psi(Q_x) \right) dx,  
\end{equation}
with $Q_0 = P_0 = C$,
and $\mathrm{diag}(Q_x)$ denoting the diagonal matrix with $ii$-entry $(Q_x)_{ii}$. We will hold off on the derivation of \eqref{matrixriccati2}. One can immediately deduce that $Q_x \sim Q_x^{\dagger}$ and hence $Q_x$ has $r$ real and almost surely distinct eigenvalues $q_1, \dots, q_r$. So, at least $F'(x_0)$ being singular for $x_0 \in [0,L]$ is equivalent to some $q_i(x_0) = 0$. Furthermore, while \eqref{matrixriccati2} is not transparently  invariant under ``rotations" (except in the case $\beta=1$), the eigenvalues form their own Markov process governed by
\begin{equation}
\label{eigdiffusion}
   d q_i = \stb q_i db_i +  \left( \psi(q_i) +  q_i  \sum_{k \neq i}  \frac{ q_i + q_k}{ q_i - q_k }  \right) dx, 
\end{equation}
with $b_1, \dots, b_r$ (new) independent standard Brownian motions. 

The derivation of \eqref{eigdiffusion} is also put off to later, but taking it for granted the proof is finished as follows. A standard calculation shows the interaction term $  q_i  \sum_{k \neq i}  \frac{ q_i + q_k}{ q_i - q_k }$ precludes intersections among the $q_i$'s, see for example \cite[Proposition 4.3.5]{AGZ}.\footnote{This reference treats the standard Dyson Brownian motion with interaction term $\sum_{k \neq i}\frac{2}{q_i-q_k}$, but the considerations are basically identical.}  So, choosing by convention  $q_i(0) = c_i$ with 
$c_1 \ge c_2 \ge  etc.$ the eigenvalues of $C$, the $q_i$ remain ordered until the possible explosion of $q_r$ to $-\infty$. Next, the presence of the term $-\lambda e^{-rx}$ in each drift term (recall the definition of $\psi$) shows that  $q_i$ are path-wise decreasing in $\lambda$. This means that new zeros among any of $q_i$ on $[0,L]$ occur at the endpoint $L$ and move from right to left. The same then holds for for the singular points of $F'$, and so the eigenvalues of $\mathfrak{G}_L$. (Thus, after the fact we see that $F$ cannot become singular before the first singular point of $F'$.)

This establishes the correspondence between $F_k$ and the probability of the appropriate number of vanishings 
 among the family $\mathbf{q}_x = \{ q_i \}_{i=1, \dots, r}$, the substitution $\lambda = e^{-\mu}$ having the same effect as in dimension one of shifting the starting time from $x=0$ to $x=\mu$. The rest is also much the same as in dimension one: the partial differential operator appearing in \eqref{rrPDE} is the space-time generator of $(x, \mathbf{q}_x)$, uniqueness is proved by an identical martingale argument, and so on.

Finally we return to  \eqref{matrixriccati2} and \eqref{eigdiffusion}.  For the former, It\^o's formula gives
\begin{align*}
    (d \A_x^{\dagger} ) P_x \A_x^{-\dagger}  & =   - dB_x^{\dagger} Q_x + (- \frac{a}{2} + \frac{1}{2 \beta} ) Q_x dx,  \\
    \A_x^{\dagger}  ( d P_x  ) \A_x^{-\dagger} & =  (dB_x + dB_x^{\dagger} ) Q_x + (r-1+ L(Q_x) ) dx,  \\
    \A_x^{\dagger} P_x (d \A_x^{-\dagger}) & =   Q_x dB_x^{\dagger} + (\frac{a}{2}+ \frac{1}{2 \beta} ) Q_x dx,
\end{align*}
and, at second order,
\begin{align*}
  (d \A_x^{\dagger}  dP_x )  \A_x^{-\dagger}  & =   - dB_x^{\dagger}   (dB_x + dB_x^{\dagger} )   Q_x =  - (r-1 +\frac{2}{\beta}) Q_x  dx, \\
    \A_x^{\dagger}  ( d P_x   d \A_x^{-\dagger} ) + (d  \A_x^{\dagger} ) P_x (d \A_x^{-\dagger}) & =  dB_x Q_x  dB_x^{\dagger}
    = \mathrm{tr} (Q_x) I_r + (\frac{1}{\beta}-1) \mathrm{diag}(Q_x),
\end{align*}
Summing the right hand sides produces \eqref{matrixriccati2}.  

For \eqref{eigdiffusion} we start by introducing some more notation, rewriting \eqref{matrixriccati2} as
$$
   d Q_x =  d B_x Q + Q dB_x^{\dagger} + F(Q_x) dx,
\quad
\mbox{ with } F(Q) =   \psi(Q) + \tr(Q) - \frac{1}{\beta} Q + ( \frac{1}{\beta} -1) \mbox{diag}(Q).
$$
The approach is classical, taking differentials throughout the spectral decomposition $Q_x = U_x^{\dagger} \mathbf{q}_x U_x$ (using now $\mathbf{q}$ for the diagonal matrix of eigenvalues $q_1, q_2 \dots$). 
To be concrete, we carry things out for $\beta=2$.\footnote{This computation is more or less reproduced in \cite{RV}.
 It should be clear afterwards how everything goes through in the quaternion case (in the real case there are no subtleties $-$ the $\mbox{diag}(Q_x)$ term is not present and the Brownian motion is isotropic).}

Set
$$
d\tilde{B}_x = U_x^{} d B_x^{} U^{\dagger}_x, \quad \mbox{ and }  \quad d U_x U_x^{\dagger} = dN_x + dG_x.
$$
Note that $\tilde{B}_x$ is not simply a copy of $dB$ (as it is for $\beta = 1$). The second definition is
a Doob-Meyer decomposition, with $N_x$ a local martingale and $G_x$ of  finite variation. 
Further,
\begin{equation}
\label{NGRules}
   d N_x^{\dagger} = - d N_x, \quad \mbox{ and } \quad dG_x^{} = - \frac{1}{2} dN_x^{} dN_x^{\dagger}, 
\end{equation}
since $d (U^{\dagger} U) = 0$.
With this in hand, suppressing the time index, an application of It\^{o}'s formula in $d \mathbf{q} = d (U Q U^{\dagger})$ produces 
\begin{align}
\label{EigSystem}
 d \mathbf{q} =  & \, d \tilde{B} \mathbf{q}    +   \mathbf{q} d \tilde{B}^{\dagger} +  U F(U^{\dagger} \mathbf{q} U )  U^{\dagger} + (d N \mathbf{q} + \mathbf{q} d N^\dagger)
     +  (d G \mathbf{q} + \mathbf{q}  d G^\dagger) \\
       & \,   + dN  \mathbf{q} d N^{\dagger} + dN d \tilde{B} \mathbf{q} + dN \mathbf{q}  d \tilde{B}^{\dagger} + d\tilde{B} \mathbf{q} d N^{\dagger} 
             + \mathbf{q} d \tilde{B}^{\dagger} d N^{\dagger}.  \nonumber
\end{align}
And as  martingale component of the right hand side must vanish off the diagonal it follows that
\begin{equation}
\label{TheN}
    dN_{ik} = \frac{  q_k  d  \tilde{B}_{ik}  + q_i d \tilde{B}^{\dagger}_{ik}}{q_i - q_k}
                   =  \frac{  q_k  d  \tilde{B}_{ik}  + q_i  d \overline{ \tilde{B}}_{ki}}{q_i - q_k}, 
\end{equation}      
for $i \neq k$.

Next, write out \eqref{EigSystem} on the diagonal  with the help of
\eqref{NGRules}:
\begin{align}
\label{EigSystemA}
    d q_i   =  &   \,    q_i ( d  \tilde{B}_{ii} +  d \tilde{B}_{ii}^{\dagger} )  +  F_{ii}(\mathbf{q}, U)  \\
                                  & \, + \sum_{k \neq i} ( q_k - q_i ) d N_{ik} d \overline{ N}_{ik}  \nonumber \\
                                    &  + \sum_{k \neq i} ( q_i d  \tilde{B}_{ki} + q_k d \overline{  \tilde{B}}_{ik} ) dN_{ik}
                                     \,  + \sum_{k \neq i} ( q_i d \overline{  \tilde{B}}_{ki} + q_k  {d  \tilde{B}_{ik}} ) 
                                     d \overline{ N}_{ik}. \nonumber
\end{align}
Here
\begin{align}
\label{diagdrift}
  F_{ii}(\mathbf{q}, U) & =  \psi(q_i) +  \sum_{k \neq i} q_k + \frac{1}{2} q_i - \frac{1}{2} 
  \sum_{k = 1}^{r}  
  \sum_{\ell=1}^r  q_k
     | u_{i \ell} |^2 | u_{k \ell}|^2, \\
                                     & = \psi(q_i) +  \sum_{k \neq i} q_k + \frac{1}{2}  \sum_{k \neq i} (q_i - q_k)  \sum_{\ell=1}^r  
     | u_{i \ell} |^2 | u_{k \ell}|^2, \nonumber
\end{align}
and $u_{ab}$  are entires of $U$.
 
 We readily have that
$
        \tilde{B}_{ii} +   \tilde{B}_{ii}^{\dagger} $ is equal in law to
$\sqrt{2}$ times a standard Brownian motion.   The higher order multiplication table is more complicated:   
$$
   d \tilde{B}_{ik}  d \overline{ \tilde{B}}_{ik} =   d \tilde{B}_{ki} d  \overline{ \tilde{B}}_{ki} = 1 -\frac{1}{2} \sum_{\ell=1}^r 
     | u_{i \ell} |^2 | u_{k \ell}|^2,     
$$
and 
$$
    d \tilde{B}_{ik}  d { \tilde{B}_{ki}} =   d\overline{ \tilde{B}}_{ik} d  \overline{ \tilde{B}}_{ki} = \frac{1}{2} \sum_{\ell=1}^r 
     | u_{i \ell} |^2 | u_{k \ell}|^2.
$$
Invoking these facts along with \eqref{TheN}  in \eqref{EigSystemA}  shows the final two lines 
of that equation are equal to
\begin{align*}
   \sum_{k \neq i} (q_i - q_k)^{-1}  &  \Bigl( q_i^2 | d \tilde{B}_{ki}|^2 + q_k^2 | d \tilde{B}_{ik}^2|^2
                                   + q_i q_k (  d \tilde{B}_{ik} {d \tilde{B}_{ki}} +    d \overline{ \tilde{B}}_{ik}  d \overline{ \tilde{B}}_{ki} ) \Bigr) \\
                                  =  &  \sum_{k \neq i}  \frac{ q_i^2  + q_k^2}{q_i - q_k} - \frac{1}{2} 
                                   \sum_{k \neq i} (q_i - q_k)  \sum_{\ell=1}^r 
     | u_{i \ell} |^2 | u_{k \ell}|^2.
\end{align*} 
When added to \eqref{diagdrift}, the second term on the right hand side above cancels the non-invariant drift term in that expression. The derivation is completed by noticing that the first term here and second term in \eqref{diagdrift} combine to produce $ \sum_{k \neq i} 
(  q_k  + \frac{ q_i^2  + q_k^2}{q_i - q_k} ) = q_i \sum_{k \neq i} 
\frac{q_i + q_k}{q_i - q_k}$.
\end{proof}

\begin{proof}[Proof of Corollary \ref{cor:Large_a}] The point is to show that for either of the processes \eqref{r1q} or \eqref{rrq}, that, given a passage to zero has occurred, explosion to $-\infty$ occurs in finite time with probability one.

For $r=1$, or the \eqref{r1q} process, this was already noticed and proved in \cite{RRZ}.  To summarize the main idea, first note that if $q_x$ hits zero it immediately becomes and stays negative for all time. Then, on the event $\tau_0 < \infty$ (with $\tau_0$ that hitting time)  the change of variables $u_x = \log (-q_{x+\tau_0})$ yields a process on the whole line, satisfying
$$
   d u_x = \frac{2}{\sqrt{\beta}} db_x + [ a + e^{u_x} + e^{-\tau_0-x - u_x}] dx, \quad r_0 \in (-\infty, \infty). 
$$
And given the sign change, one wants to show that $u_x$ hits $+\infty$ in finite time. But $u_x$ is path-wise lower bounded by the the homogeneous process $v_x$ defined by $  d v_x = \frac{2}{\sqrt{\beta}} db_x + [ a + e^{v_x}] dx$ (run on the same Brownian motion). 
Feller's test can now be applied to show $ \tau_{+\infty}(s) < \infty$ with probability one as long as $a \ge 0$.

For $r=1$ the condition on $a$ can be shown to be sharp, and we expect that for all $r$ the $a\ge 0$ is the right condition. In any event, consider \eqref{rrq} and notice again that if the lowest point $q_{r}$ vanishes at some $x' < \infty$, then $q_{r,x} < 0$ for all $x > x'$. With any other the other particles $q_i$, $i = 1, \dots, r-1$, possibly still positive, the corresponding contribution to the drift $q_r \times  \frac{q_r+ q_i}{q_r -q_i}$ may also be positive, and so push $q_r$ in the ``wrong" direction. On the other hand, in that case one has that $\left|   \frac{q_r+ q_i}{q_r -q_i} \right| \le 1$. This means that, at the very least, after $q_r$  hits zero it is bounded above by the process
$$
  d z_x = \stb z_x db +  \left(  (a  - (r-1) + \frac{2}{\beta}) z_x - z_x^2 - e^{-r x} \right) dx,
$$ 
ignoring shifts in the time-like variable. The proof is finished by noting that the $r=1$ argument can be applied to $z_x$ with $a$  replaced by $a -r +1$.
\end{proof}


\section{Hard-to-soft transition}
\label{sec:HardSoft}

By following the law(s) on paths induced by \eqref{r1q} and \eqref{rrq} in the $a \rightarrow \infty$ limit we show that (resealed) hard edge spiked laws degenerate to the soft edge spiked laws (Theorem \ref{thm:HardSoft}).

\begin{proof}[Proof of Theorem \ref{thm:HardSoft}] Since we are taking an $a \rightarrow \infty$ limit, Corollary \ref{cor:Large_a} shows that from the start 
we can just track blow-ups (to $-\infty$) for the hard-edge processes.

Taking first $r=1$, introduce the time-space scaling,
\begin{equation}
\label{pathsub}
    \eta(x) =  a^{-2/3} {q}( a^{-2/3} x) - a^{1/3},  
\end{equation}
where $q$ is defined by \eqref{r1q} after taking $a$ into $2a$: $dq = \stb q db + ((2a+\frac{2}{\beta}) q - q^2 - e^{-x}) dx$. With the hard-edge spectral parameter scaled as in $a^2 + a^{4/3} \lambda$, the process $\eta$ is then started at time $x_0 = - a^{2/3} \log (1- a^{-2/3} \lambda)$ from $\eta_0 = \eta(x_0) = a^{-2/3}(c-a)$.  For this recall that the start time $\mu$ for $q$ is related to the spectral parameter via $\lambda = e^{-\mu}$.

Now the essential observation is the same as in the corresponding proof from \cite{RR} for the limit of $\Lambda(\beta, 2a  , \infty)$:
$q$ hits $-\infty$ in finite time if and only if $\eta$ does. Hence, things come down to showing that we have the functional convergence of
$\eta = \eta(x, a)$ governed by
\begin{align}
\label{scaledr1diffusion}
   d \eta_x   = & \, \stb (1 + a^{-1/3} \eta_x) db_x +  [ a^{2/3} (1- e^{-a^{-2/3} x}) - \eta_x^2 +\tb (a^{-1/3} + a^{-2/3} \eta_x) ] dx \\
                  := & \, e(x, \eta,a) db_x + f(x, \eta, a) dx,  \nonumber                        
\end{align}
with  initial state, 
$$
   (x_0(a), \eta_0(a)) =  (- a^{2/3} \log (1- a^{-2/3} \lambda), \, a^{-2/3}c(a) -a^{1/3}), 
$$
and a new Brownian motion $b_x$,
to the one-spike soft edge process \eqref{softdiffusion1}:
\begin{align}
   d p_x  = & \,  \stb db   + (x- p_x^2) dx, 
\end{align}
with initial state $(\lambda, w)$.
But now,
$$
   e(x, \eta,a) \rightarrow 1, \quad f(x,\eta,a) \rightarrow x - \eta^2,
$$
uniformly on compact sets in $(x, \eta)$, while $x_0  = - a^{2/3} \log (1- a^{-2/3} \lambda) \rightarrow \lambda$ and 
$\eta_0 =  a^{-2/3}c - a^{1/3} \rightarrow w$ by  assumption. 
This is sufficient to imply the convergence of the solutions of the martingale problems ($\PP_{x_0, \eta_0}^{\eta} \Rightarrow 
\PP_{\lambda, w}^p$), see for example Theorem 11.3.3 of \cite{SV}.\footnote{Here we are tacitly assuming that the limiting starting point 
$w$ is finite, if $a^{-2/3}c - a^{1/3} \rightarrow \infty$ an additional approximation step is needed, but this has been carried out in \cite{RR}.}
In this setting, this is the same as convergence in law on paths in the uniform-on-compacts  topology.

From here the proof that $a^{2/3}- a^{-4/3} \Lambda_0(\beta, 2a ,c(a)) \Rightarrow  \lambda_0(\beta, w)$ will follow  from the fact that
$\tau_{-\infty} (\eta_a) \Rightarrow \tau_{-\infty}(p)$, where $\tau_{b}$ denotes the passage time to $b$.
 Again, this is what was done in \cite{RR} in the case that  $c(a) = w = +\infty$. If $b$ were finite the convergence in law of the corresponding passage times is a consequence of the continuous mapping theorem, and so it is enough to show that
 \begin{equation}
 \label{ApproxExplosionTime}
    \lim_{b \rightarrow -\infty} \PP_{\lambda, w} ( \tau_{-\infty}(p) - \tau_b(p)  > \epsilon)  \vee \sup_{a \gg 1} \PP_{x_0, \eta_0}
     ( \tau_{-\infty}(\eta) - \tau_{b}(\eta) | > \epsilon) = 0
 \end{equation}
 for any $\epsilon > 0$. Since this approximation step was swept under the carpet in \cite{RR} we take the opportunity here to provide some details.  
 
 Consider the question for the $p$-diffusion (the same argument will work for $\eta$, but is just more cumbersome as one needs to keep track of the parameter $a \rightarrow \infty$). Make a further approximation and truncate the stopping times, replacing $\tau_b(p)$ with 
 $\tau_b^k(p) = \tau_b(p) \wedge k$, noting that $\tau_b(\eta) \Rightarrow \tau_b(p)$ if and only if $\tau_b^k(\eta) \Rightarrow \tau_b^k(p)$. 
 Also introduce the time $\gamma_b^k(p) = \inf\{ x > \tau_b^k(p) : p(x) > - \frac{4}{\beta} k \}$.  Now take $f$ to be a smooth function
 with $f(x) = 1/x$ for $x < -10$ (say). It\^{o}'s lemma along with the optional stopping theorem gives: assuming $k > 10 \vee 4/\beta$ and
 dropping the dependence of the various random times on the path,
 \begin{align*}
   \ev  f (p(\tau_{-\infty}^k \wedge \gamma_b^k) ) & = \ev  f (p(\tau_b^k) ) + \ev  \int_{\tau_b^k}^{ \tau_{-\infty}^k \wedge \gamma_b^k} 
   ( \frac{4}{\beta} p^{-3}  - x p^{-2} + 1) dx \\
     & \ge \ev  f (p(\tau_b^k) ) + \frac{k}{k- 4/\beta} \ev [  \tau_{-\infty}^k \wedge \gamma_b^k - \tau_b^k ] .
 \end{align*}
 Hence, since either $p(\tau_b^k) = b$ or $\tau_b^k = \tau_b^{\infty}$,
 $$
    \pr \left( \tau_{-\infty}^k \wedge \gamma_b^k - \tau_b^k > \epsilon \right) \le - \frac{ k}{ b \epsilon (k - 4/\beta)} \pr \left( p(\tau_b^k) = b \right).
 $$ 
 And since $\lim_{b \rightarrow - \infty} \pr ( \gamma_b^k < k, p(\tau_b^k) = b) =0$, one has that
 $$
   \lim_{b \rightarrow -\infty}  \pr  \left( \tau_{-\infty}^k - \tau_b^k > \epsilon \right) = \lim_{b \rightarrow -\infty} \pr  \left( \tau_{-\infty}^k \wedge \gamma_b^k - \tau_b^k > \epsilon \right) = 0,
 $$
 which proves \eqref{ApproxExplosionTime} (for the $p$ process).

Moving to the joint convergence  of $\Lambda_{0}, \Lambda_1,...$, we have to show that, for any fixed $k$ and $a \rightarrow \infty$
$$
   \PP_{x_0(a), \eta_0(a)} \left( x \mapsto \eta(x,a) \mbox{ explodes exactly } k  \mbox{ times}  \right) 
$$  
goes over into the  $\PP_{\lambda, w}$-probability of the same event for $p$. Recalling the shorthand
$\kappa_{x, b} (dy)= \PP_{x,b} (\tau_{\infty} \in dy)$ for the passage time distribution of either process, the 
obvious thing to do is rewrite the above as in 
$$
    \int_{ x_0 < x_1 < \cdots < x_k}  \kappa_{(x_0(a), \eta_0(a))}(dx_1) \kappa_{(x_1, +\infty)} (dx_2) \dots \kappa_{(x_{k}, +\infty)}(\{+\infty\}),
$$
and use the $k=0$ result. In particular, we know, now emphasizing the processes involved, that $ \kappa_{(x_0(a), \eta_0(a))}^{\eta(a)}(dx) $ tends 
weakly to $\kappa_{(\lambda, w)}^p(dx)$ in the sense of measures. Again this includes  $  \lim_{a \rightarrow \infty} \kappa_{(x_0(a), \eta_0(a))}^{\eta(a)}(\{+\infty\}) =  \kappa_{(\lambda, w)}^p(\{+\infty\})$, which is all that was needed for the first step. The same argument
gives convergence of the passage time distributions starting from fixed initial conditions, or that 
$ \kappa_{(x_0, +\infty)}^{\eta(a)}(dx) \rightarrow   \kappa_{(x_0, +\infty)}^p(dx)$ weakly. Once more, technical issues associated with starting from 
$+\infty$ 
 were dealt with in \cite{RR}: since $+\infty$  is an entrance point for all $\eta(a)$ and $p$, one can well approximate the distributions(s) by the same process started at some (large) finite point.

Consider  $k=1$, which enough to explain what needs to be done.  The issue is the convergence of the integral
$ \int \mu_n(dx) f_n(x)$ to its intended limit 
in which: (1) the (sub-probability) measures $\mu_n$ are absolutely continuous and converge weakly to a (sub-probability) measure $\mu $ and,  (2) $f_n \rightarrow f$ pointwise where all $f_n$ and $f$ are uniformly bounded (by one) and continuous.  (Here $f_n(x) = \kappa_{(x, +\infty)}^{\eta(a)}(\{+\infty\})$ abd $ \mu_n(dx)  = \kappa_{(x_0(a), \eta_0(a))}^{\eta(a)}(dx)$. Outside of $\{ + \infty \}$ the measures $\mu_n$ and $ \mu$ posses no point masses.) By the tightness of $\mu_n$ you can cut down the domain of integration to a bounded interval $I$.  After this one can invoke say Egorov's theorem along with the absolute continuity of the sequence of measures to deal with the error term $\int_I (f_n - f) \mu_n(dx)$.

When $r>1$ the only novelty in establishing the process level convergence is the interaction terms. Following \eqref{pathsub} we scale $\mathbf{q}$ as in
$$
   \eta_i(x) = a^{-2/3} q_i(a^{-2/3} x) - a^{1/3}, \quad  i = 1, \dots, r,
$$
and \eqref{scaledr1diffusion} is replaced by: for a given index $i$, 
\begin{align*}
d\eta_{i,x}   =  \, \stb (1 + a^{-1/3} \eta_{i,x}) db_{i,x}  & +   [ a^{2/3} (1- e^{-ra^{-2/3} x}) - \eta_x^2 +\tb (a^{-1/3} + a^{-2/3} \eta_{i,x}) ] dx \\
     & +    ( a^{-1/3} + a^{-2/3} \eta_{i, x} ) \sum_{j,  \neq i} \frac{ 2 a^{1/3} + \eta_{i,x} + \eta_{j,x}}{\eta_{i,x}-\eta_{j,x}} dx.
\end{align*}
Again the diffusion coefficient and first drift coefficient converge uniformly on compact sets to the right objects $-$ the only change is the shift
from $e^{-a^{-2/3} x}$ to $e^{-ra^{-2/3}}$ in the latter. And again the initial conditions settle down the claimed limits, more or less by design.
While the final term does tend pointwise to $\sum_{j \neq i} \frac{2}{\eta_i - \eta_j}$, which is the interaction term in \eqref{softdiffusion2}, 
the convergence is certainly not uniform over $\{ \eta_1> \eta_2> \cdots > \eta_r \}$. 

 On the other hand, one does have convergence on paths up to the stopping time 
$$
\tau_{\epsilon} := \{ \inf{x} :  \eta_{i, x} - \eta_{j,x} \le \epsilon \mbox{ for any } i < j \}.
$$
And since $\tau_{\epsilon} \rightarrow \infty$ in probability as $\epsilon \rightarrow 0$ for both the $\boldsymbol\eta$ and $\mathbf{p}$ processes,
general considerations (see for instance \cite[Lemma 11.1.1]{SV}) yield the convergence $\mathbb{P}^{\boldsymbol\eta} \Rightarrow \mathbb{P}^{\mathbf{p}}$. Previously we had cited \cite[Proposition 4.3.5]{AGZ} for the non-crossing property of (a process related to) $\boldsymbol{\eta}$, which proves the result for all $\beta \ge 1$. In fact, for $\beta < 1$ it is known that there is a crossing in finite time with positive probability, see 
\cite[Theorem 3.1]{CL} which again deals with a more standard Dyson Brownian motion. We may still have the path convergence between the hard and soft edge diffusions for $\beta < 1$, but this would require an additional argument.  This is the point of Remark \ref{rem:HtoS}.

\end{proof}

\section{Fredholm approach for $\beta=2$}
\label{sec:Appendix}

For completeness we include an explicit calculation of the limiting $r$-spiked distribution function for $\beta = 2$ (and integer $a \ge 0$). The reported limiting kernel has already appeared 
in \cite{DF}. The main  point here is to detail the trace norm convergence for the associated integral operators, and so deduce the corresponding Fredholm determinant 
form of the distribution. We follow the original strategy of \cite{BBP}. A second motivation is to carry out the hard-to-soft limit in this context, giving an analytic proof 
of Corollary \ref{cor:multisupcrit1}, at least for $\beta = 2$ and $a = 0, 1,\dots$, etcetera. 

Now let $X$ be a complex $n\times (n+a)$ Gaussian matrix with diagonal covariance $\Sigma = \Sigma_r \oplus I_{n-r}$ 
and $\Sigma_r  = \mbox{diag}(\sigma_1, \sigma_2, \dots, \sigma_r)$, and  let $\Lambda_{min} = \Lambda_{min}(\Sigma_r)$ denote the minimal eigenvalue
of $X X^{\dagger}$. Our starting point is a version of the following basic formula.

\begin{proposition}  (After \cite[Proposition 2.1]{BBP}) At finite $n$ there is the identity
$$
   \pr \Bigl( n \Lambda_{min}(\sigma_1, \dots , \sigma_r) > t \Bigr) = {\det}_{L^2[0,t]} ( I - \mathsf{K}_n )
$$
where $\mathsf{K_n} = \mathsf{K_{n, a, \Sigma_r}}$ is the kernel operator defined by
\begin{equation}
\label{FredK}
    \mathsf{K}_{n}(x,y) = \frac{1}{n} \int_{\mathcal{C}_{1, 1/{\boldsymbol{\sigma}}}} \frac{dz}{2 \pi i} \, \int_{\mathcal{C}_0}  \frac{dw}{2 \pi i}  \, \frac{e^{- \frac{x}{n} z + \frac{y}{n} w }}{w-z}  
    \left(\frac{z}{w} \right)^{r+a}    \left(  \frac{ 1- \frac{1}{w}}{1 - \frac{1}{z}} \right)^{n-r}  \prod_{k=1}^r \frac{\sigma_k^{-1} - w}{\sigma_k^{-1} - z}.                                            
\end{equation}
Here $\mathcal{C}_0$ and $\mathcal{C}_{1, 1/\boldsymbol{\sigma}}$ are  any disjoint, simple, closed contours (oriented counter-clockwise) which enclose zero and 
$\{ 1, \sigma_1^{-1}, \dots,
\sigma_r^{-1} \}$, respectively.
\end{proposition}

It is  easy to show that:

\begin{claim} 
\label{thm:lastthm} Setting $\sigma_k = \frac{c_k}{n}$ for  $k=1,\dots, r$ in \eqref{FredK} and taking  $n \rightarrow  \infty$,  $K_n$ converges in the trace norm of operators on $L^2[0,t] \mapsto L^2[0,t]$ to $\mathsf{K}$ defined by
\begin{equation*}
\label{FredKLimit}
  \mathsf{K}(x,y) =  \left(  \int_{\mathcal{C}_{1/\boldsymbol{c}}}  \frac{dz}{2 \pi i} \, \int_{\mathcal{C}_0}  \frac{dw}{2 \pi i} + 
   \int_{\mathcal{C}}  \frac{dz}{2 \pi i} \, \int_{\mathcal{C}^{\prime}}  \frac{dw}{2 \pi i}   \right) \mathsf{k}(x,y; z, w)
\end{equation*}
where
\begin{equation*}      
     \mathsf{k}(x,y; z,w ) =   \frac{e^{- {x} (z - \frac{1}{z}) + {y} (w - \frac{1}{w} ) }}{w-z} 
    \left(\frac{z}{w} \right)^{r+a} \prod_{k=1}^r \frac{1 -c_k w}{1  - c_k z}.
\end{equation*}
The notation/conventions are repeated from above in that   $\mathcal{C}_0$ and $\mathcal{C}_{1/\boldsymbol{c}}$ are disjoint, simple, closed contours, with  the latter 
 enclosing
${c_1^{-1}, \dots, c_r^{-1}}$. The contours $\mathcal{C}$ and $\mathcal{C}^{\prime}$ can be taken any nested disjoint contours surrounding the origin.
\end{claim}

The probabilistic content of Claim \ref{thm:lastthm} is that
\begin{equation}
      \lim_{n \rightarrow \infty} \pr \Bigl( n \Lambda_{min}({c_1}/{n}, \dots , {c_k}/{n}) > t \Bigr) = {\det}_{L^2[0,t]} ( I - \mathsf{K}).
\end{equation}
The related result in \cite{DF} establishes the pointwise convergence of ${\mathsf{K}}_n(x,y)$ to ${\mathsf{K}}(x,y)$ $-$ attention to  error estimates being beside the point for the considerations there. More notably, the authors of \cite{DF} derive different formulations
for the $\mathsf{K}$ kernel, showing for instance that it is a rank $r$ perturbation of a conjugated  Bessel kernel. This stands in analogy to the results of \cite{BBP} which
show that the $r$-spiked soft edge kernel is a rank $r$ perturbation of the Airy kernel.  Here we highlight the following consequence of 
the operator  convergence to $\mathsf{K}$. 
Compare part (b) of \cite[Theorem 1.1]{BBP}.

\begin{corollary}
\label{cor:HtSbeta2}
 Let $\mathsf{K}_{c}$ denote kernel operator $\mathsf{K}$  in the case that $c_1 = \cdots = c_r = c$. Then,
$$
   \lim_{c \downarrow 0} {\det}_{L^2[0, ct]} ( I - \mathsf{K}_{c} ) = {\det}_{L^2[0, t]} ( I -  \mathsf{L} ),
$$
where $\mathsf{L}$ is the projection onto the span of the first $r$ Laguerre functions of parameter $a$. In other words, the distributional
limit of  $\frac{1}{c} \Lambda_{min} (cI)$ as $n \rightarrow \infty$ then $c \rightarrow 0$ (in that order) is that of the 
 minimal eigenvalue of a $(r, r+a)$ complex Wishart matrix. 
\end{corollary}

\begin{proof}[Proof of the claim] Unlike in \cite{BBP} this requires no steepest descent argument, just a change of variables. 
First, replace outer integral over $\mathcal{C}_{1, 1/\boldsymbol{\sigma}}$  by that over the union of two disjoint loops $\mathcal{C}_{1} \cup \mathcal{C}_{1/\boldsymbol{\sigma}}$, each still avoiding the contour 
$\mathcal{C}_0$. This splits $\mathsf{K}_n$ into two terms. Consider that given by the 
$ \int_{\mathcal{C}_{1}} \frac{dz}{2 \pi i} \, \int_{\mathcal{C}_0} \frac{dw}{2\pi i}$ integral. Fubini plus two residue calculations show that this integral can be replaced by that say over $ \int_{ {|z| = 1+\delta}} \frac{dz}{2 \pi i} \, \int_{ {|w|= 1+ \delta'} } \frac{dw}{2\pi i}$ for $\delta' < \delta$, or any choice of nested contours enclosing 
both 0 and 1, but not any of the $\sigma_k^{-1}$. With our choice  $1+\delta \le \min \{\sigma_k^{-1} \}$ 
(as each  $\sigma_k^{-1}$ will scale like $n$, we can assume they are all greater that one). 

After this, put $\sigma_k = \frac{c_k}{n}$
in \eqref{FredK}, substitute $nz$ and $nw$ for $z$ and $w$. The contours of course move as well, but can be deformed again after the fact.  Setting,
$$
   {\mathsf{k}}_n(x,y; z,w) =  \frac{e^{- {x} z + {y} w }}{w-z}  
    \left(\frac{z}{w} \right)^{r+a}    \left(  \frac{ 1- \frac{1}{nw}}{1 - \frac{1}{nz}} \right)^{n-r}  \prod_{k=1}^r \frac{c_k^{-1} - w}{c_k^{-1} - z} ,
$$
we have that,
\begin{align*}
    \mathsf{K}_{n}(x,y) 
                                   & =    
     \int_{\mathcal{C}_{1/\mathbf{c}}} \frac{dz}{2 \pi i} \, \int_{\mathcal{C}_0} \frac{dw}{2 \pi i} \, \mathsf{k}_n (x,y; z,w)
                                + \int_{ \mathcal{C} }\frac{dz}{2 \pi i} \, \int_{ \mathcal{C}^{\prime}} \frac{dw}{2 \pi i} \, \mathsf{k}_n
                                (x,y; z,w).                             
\end{align*}
Here ${\mathcal{C}_{1/\mathbf{c}}}$ and  ${\mathcal{C}_0}$ can be chosen as in the claim,  and taken to be
independent of $n$. Similarly, one can choose $\mathcal{C} =\{ |z| = \epsilon \}$ and $\mathcal{C}' =\{ |z| = \epsilon' \}$, for 
$\epsilon' < \epsilon <  \min ( 1, \min ( c_k^{-1} ))$. 

 Now  the kernel  $k_n$ only depends on $n$ through  $\left(  \frac{ 1- \frac{1}{nw}}{1 - \frac{1}{nz}} \right)^{n-r}$,
 which converges uniformly to $e^{1/z-1/w}$ along the contours. And with $|w-z|$  bounded below for $(z,w)$ restricted to either set of contours,
  one sees that
  $\textsf{K}_n(x,y)  \rightarrow \textsf{K}(x,y)$ uniformly for $(x,y) \in [0,t]^2$. But then,
  $
  \int_0^t \textsf{K}_n(x,x) dx \rightarrow     \int_0^t \textsf{K}(x,x) dx,$ 
  and
  $   \int_0^t \int_0^t f(x) g(y) \textsf{K}_n(x,y) dx dy$ 
  $\rightarrow$   
  $ \int_0^t \int_0^t f(x) g(y)   \textsf{K}(x,y) 
  dx dy,$
for all $f,g \in L^2[0,t]$. This implies the trace norm convergence by \cite[Theorem 2.20]{Simon}.
\end{proof}

\begin{proof}[Proof of Corollary \ref{cor:HtSbeta2}] After scaling we can work with the kernel operator 
\begin{align}
\label{eq:cK}
  c \mathsf{K}_c(cx,cy) & =      \left( \int_{\mathcal{C}_{1/c}} \, \int_{\mathcal{C}_0}   + 
                                                          \int_{|z|=\epsilon}  \, \int_{|w| = \epsilon/2}  \right) c k(cx, cy; z, w) \, \frac{dz}{2 \pi i} \frac{dw}{2 \pi i}  \\
                                     &  =      \left( \int_{\mathcal{C}_{1}} \, \int_{\mathcal{C}_0}  + 
                                                          \int_{|z|=\epsilon}  \, \int_{|w| = \epsilon/2}  \right)               
              \frac{e^{- {x} (z - \frac{c}{z}) + {y} (w - \frac{c}{w} ) }}{w-z} 
    \left(\frac{z}{w} \right)^{r+a} \left( \frac{1 - w}{1  -  z} \right)^r  \,  \frac{dz}{2 \pi i} \frac{dw}{2 \pi i}, \nonumber
\end{align}
restricted to $L^2[0,t]$.
This is just another change simple change of variables $-$ $(z, w) \mapsto (\frac{z}{c}, \frac{w}{c})$ $-$  followed by a pair of contour deformations. Here we have chosen a particular $\mathcal{C}$ and $\mathcal{C}^{\prime}$ in the second integral (and deformed back to this choice after the variable change).

Next, by the same convergence criteria used for Theorem \ref{thm:lastthm}, and basically the same argument,  we see that \eqref{eq:cK} tends in trace norm to its evaluation at $c=0$. Letting $\epsilon \rightarrow 0$ in that case shows that the second integral in fact drops out. That is, $\det_{L^2[0,ct]} (I - \mathsf{K}_c) \rightarrow \det_{L^2[0,t]} (I - \textsf{L})$, where
\begin{equation}
\label{LKernel}
   \mathsf{L}(x,y) = \int_{\mathcal{C}_1} \frac{dz}{2 \pi i} \, \int_{\mathcal{C}_0}  \frac{dw}{2 \pi i}  \, \frac{e^{- {x} z  + {y} w  }}{w-z} 
    \left(\frac{z}{w} \right)^{r+a} \left( \frac{1 - w}{1  -  z} \right)^r.
\end{equation}
At this point one may use integral representations for the Laguerre polynomials
(for example, $L_k^a(x)= \int_{\mathcal{C}_0} \frac{dz}{2 \pi i} \frac{e^{-xz}}{z^{k+1}} (1+z)^{k+a}$) to conclude that indeed
$\mathsf{L}(x,y) = e^{-x/2} e^{-y/2} \sum_{k=0}^{r-1} L_k^a(x) L_k^a(y)$. On the other hand, the original kernel $\textsf{K}_n$, or \eqref{FredK}, corresponds to the (null) Wishart case if all the $\sigma_k$ are set to one. That is, it must degenerate
to the Laguerre projection kernel (or some conjugation thereof). But, undoing the scaling (replacing $\mathsf{K}_n(x,y)$ with $n \mathsf{K}_n(nx,ny)$), then setting $n=r$ and  $\sigma_k=1$ for $k=1,\dots, r$, we recognize this is the same formula as the right hand side of \eqref{LKernel}.
\end{proof}

\bigskip

\noindent
{\bf{Acknowledgements.}}  Many thanks to Tom Kurtz, Benedek Valk\'o, and Ofer Zeitouni for helpful discussions. B.R.~was supported in part by NSF grant DMS-1340489 and grant 229249 from the Simons Foundation.

\end{document}